\documentclass[11pt,letterpaper]{amsart}
\usepackage{mathrsfs,color,bm,amsmath,amsfonts,amssymb,dsfont,amscd,extarrows,enumerate,verbatim}
\usepackage[all]{xy}
\numberwithin{equation}{section}

\newtheorem{thm}{Theorem}[section]
\newtheorem{cor}[thm]{Corollary}
\newtheorem{lem}[thm]{Lemma}
\newtheorem{remark}[thm]{Remark}

\newtheorem{defn}[thm]{Definition}

\newcommand{\mr}{\mathbb{R}}
\newcommand{\mb}{\mathcal{B}}
\newcommand{\mc}{\mathbb{C}}
\newcommand{\mo}{\mathcal{O}}
\newcommand{\rw}{\rightarrow}
\newcommand{\rwo}{\mapsto}

\newcommand{\nve}{\vec{\bm n}}
\newcommand{\distr}{\operatorname{distr.}}
\newcommand{\vol}{\operatorname{Vol}}
\newcommand{\sgn}{\operatorname{sgn}}
\newcommand{\ave}{\operatorname{ave}}
\newcommand{\h}{\operatorname{H}}
\newcommand{\sh}{\operatorname{SH}}
\newcommand{\shp}{\operatorname{SHP}}

\DeclareMathOperator{\loc}{loc}

\DeclareMathOperator{\lip}{Lip}
\DeclareMathOperator{\supp}{supp}
\DeclareMathOperator{\diag}{diag}
\DeclareMathOperator{\diam}{diam}
\DeclareMathOperator{\dom}{Dom}
\usepackage[pagebackref=true,colorlinks=true,bookmarksopen,bookmarksnumbered,citecolor=red, linkcolor=blue, urlcolor=cyan]{hyperref}

\allowdisplaybreaks[4]
\begin{document}

\title[Uniform estimates of Green functions and Sobolev inequalities]
{Uniform estimates of Green functions and Sobolev-type inequalities on real and complex manifolds}

\author[F. Deng]{Fusheng Deng}
\address{Fusheng Deng: \ School of Mathematical Sciences, University of Chinese Academy of Sciences \\ Beijing 100049, P. R. China}
\email{fshdeng@ucas.ac.cn}

\author[G. Huang]{Gang Huang }
\address{Gang Huang: \ School of Mathematical Sciences, University of Chinese Academy of Sciences \\ Beijing 100049, P. R. China}
\email{huanggang21@mails.ucas.ac.cn}

\author[X. Qin]{Xiangsen Qin}
\address{Xiangsen Qin: \ School of Mathematical Sciences, University of Chinese Academy of Sciences \\ Beijing 100049, P. R. China}
\email{qinxiangsen19@mails.ucas.ac.cn}

\begin{abstract}
     We prove certain  $L^p$ Sobolev-type and Poincar\'e-type inequalities for functions on real and complex manifolds 
    for the gradient operator $\nabla$, the Laplace operator $\Delta$, and the operator $\bar\partial$. 
    Integral representations for functions are key to get such inequalities. 
   The proofs of the main results involves certain uniform estimates for the Green functions and their gradients on Riemannian manifolds,
   which are also established in the present work. 
\end{abstract}

\maketitle
\tableofcontents
    \section{Introduction}
    This article is a continuation and extension of \cite{DJQ}. 
    That paper considers the $\bar\partial$-Poincar\'e inequalities
    in an open subset of $\mc^n$. 
    In the present work, we consider some  $L^p$ Sobolev-type and Poincar\'e-type inequalities on manifolds 
    for the gradient operator $\nabla$, Laplace operator $\Delta$ or operator $\bar\partial$. 
    As observed in \cite{DJQ}, we can use  integral representations for functions as the key tool to get such inequalities. 
    In order to make such an idea work,  we need certain uniform estimates for Green functions of the Laplace operator and their gradients on Riemannian manifolds. 
    This is established in the following 
    \begin{thm}\label{green function uniform}
     Let $(\overline M=M\cup\partial M,g)$ be a compact Riemannian  manifold with boundary (of dimension $n\geq 2$), and let $G$ be the Green function for the Laplace operator on $M$.
     Then $G$ has the following uniform estimate:
     \begin{itemize}
     \item[(i)] If $n\geq 3$, there is a constant $C:=C(\overline M,g)>0$ such that 
     $$ |G(x,y)|\leq Cd(x,y)^{2-n},$$ $$ |\nabla_y G(x,y)|\leq Cd(x,y)^{1-n},$$
     for all $(x,y)\in (\overline M\times\overline M)\setminus\diag$.
     \item[(ii)] If $n=2$, there is a constant $C:=C(\overline M,g)>0$ such that 
     $$ |G(x,y)|\leq C(1+|\ln d(x,y)|),$$
     $$|\nabla_y G(x,y)|\leq C(1+|\ln d(x,y)|)\frac{1}{d(x,y)},$$
     for all $(x,y)\in (\overline M\times\overline M)\setminus\diag$.
     \end{itemize}
     \end{thm}
     \indent Please see Subsection \ref{notations} for various notations and conventions that will be used in this article. 
     When we say $(M,g)$ (resp. $(\overline M,g)$) is a Riemannian manifold, we mean it is a connected oriented smooth manifold without boundary (resp. with smooth boundary) 
     and it is equipped with a smooth Riemannian metric (resp. a Riemannian metric which is smooth up to the boundary). 
     We make such smoothness convention just for convenience, and many results of this paper also hold under weaker regularity assumption, 
     but we don't pursue this here.  Moreover, we assume all manifolds are of real dimension $n\geq 2$. 
     
     Some weaker forms of Theorem \ref{green function uniform} are already known. 
     It is well known for bounded domains in $\mr^n$ with $n\geq 3$, see \cite[Theorem 3.3]{Widman} for example..
  Similar result is stated in  \cite[Theorem 4.17]{Aubin}, but with the constant $C$ depending on the distance from $x$ to the boundary. 
     It seems that the above uniform estimates cannot be simply obtained through a minor modification of the proof of \cite[Theorem 4.17]{Aubin}. 

Following the basic insight in \cite{DJQ}, we can derive from Theorem \ref{green function uniform} certain $L^p$ Sobolev-type inequalities for the gradient operator and the Laplace operator.

    \begin{thm}\label{thm:graident boundary}
    Let $(\overline M,g)$ be a compact Riemannian manifold with boundary, and let $p,q,r$ satisfy
    $$
         1\leq p<\infty,\ 1\leq q<\infty,\ 1\leq r<\infty,\ q(n-p)<np,\ q(n-1)<nr.
    $$
    Then there is a constant $\delta:=\delta(\overline M,g,p,q,r)>0$ such that
    $$
    \delta\|f\|_{L^q(M)}\leq \|\nabla f\|_{L^p(M)}+\|f\|_{L^r(\partial M)},\ \forall f\in C^1(\overline M).$$
    \end{thm}
    \begin{thm}\label{thm:laplace boundary}
    Suppose $(\overline M,g)$ is a compact Riemannian manifold  with boundary. Let $p,q,r$ satisfy
    $$
         1\leq p<\infty,\ 1\leq q<\infty,\ 1\leq r<\infty,\ q(n-2p)<np,\ q(n-1)<nr,
    $$
    and if $n=2$, we moreover require $p>1$ if $q>1$. Then there is a constant $\delta:=\delta(\overline M,g,p,q,r)>0$ such that
    $$
    \delta\|f\|_{L^q(M)}\leq \|\Delta f\|_{L^p(M)}+\|f\|_{L^r(\partial M)}$$
    for all $f\in C^0(\overline M)\cap C^2(M)$ such that $\Delta f$ has a continuous continuation to the boundary.
    \end{thm}
     Clearly, we can take $1\leq p=q=r<\infty$ in Theorem \ref{thm:graident boundary}
     and Theorem \ref{thm:laplace boundary}. If $M\subset\mr^n$ is a bounded open subset and $p<n$, 
     then Theorem \ref{thm:graident boundary} can be seen as a generalization of \cite[Theorem 1.2]{maggi} (see also \cite[Formula (2.1)]{Anv}) 
     if we overlook the estimates for the constants in the inequalities, 
     and it is also a generalization of \cite[Corollary 1.5]{DJQ}. 
     Theorem \ref{thm:laplace boundary} still hold when $p=q=r=\infty$ and $M\subset\mr^n$ is a bounded open subset, see Theorem 6.2 of \cite[Chapter 2]{chen}. 
     But we don't know whether Theorem \ref{thm:graident boundary} holds or not if $p=q=r=\infty$. 
     
     We have the following consequence from Theorem \ref{thm:laplace boundary} for harmonic functions.
    \begin{cor}\label{cor:laplace boundary}
    Let $(\overline M,g)$ be a compact Riemannian manifold with boundary, and let $p,q$ satisfy
    $$
         1\leq p<\infty,\ 1\leq q<\infty,\ q(n-1)<np.
    $$
    Then there is a constant $\delta:=\delta(\overline M,g,p,q)>0$ such that
    $$
    \delta\|f\|_{L^q(M)}\leq \|f\|_{L^p(\partial M)}$$
    for any harmonic function $f$ on $M$ with $f\in C^0(\overline M)$.
    \end{cor}
    \indent When $p=q=2$, Corollary \ref{cor:laplace boundary} follows from \cite[Proposition 1.8, Chapter 5]{taylor}.

With similar idea, we can get Poincar\'e-type inequalities for the gradient operator and Laplacian on compact Riemannian manifolds without boundary. 

     \begin{thm}\label{thm:poincare compact}
    Let $(M,g)$ be a compact Riemannian manifold without boundary, and let $p,q$ satisfy
    $$
         1\leq p<\infty,\ 1\leq q<\infty,\ q(n-p)\leq np.
    $$
     There is a constant $\delta:=\delta(M,g,p,q)>0$ such that
    $$
    \delta\|f-f_{\operatorname{ave}}\|_{L^q(M)}\leq \|\nabla f\|_{L^p(M)},\ \forall f\in C^1(M),$$
    where 
    $$f_{\ave}:=\frac{1}{|M|}\int_{M}fdV.$$
     Furthermore, when $p>n$, there is a constant $\delta:=\delta(M,g,p)>0$ such that 
    $$\delta\|f-f_{\operatorname{ave}}\|_{L^\infty(M)}\leq \|\nabla f\|_{L^p(M)},\ \forall f\in C^1(M).$$ 
    \end{thm}
     When $1\leq p< n$, Theorem \ref{thm:poincare compact} has been proved in \cite{Hebey}, see Theorem 2.11 there. 
     For the other cases, one may modify the proof of Theorem \ref{thm:graident boundary} a little to get them.
     \begin{thm}\label{thm:laplace compact1}
    Let $(M,g)$ be a compact Riemannian manifold without boundary. Assume $p,q$ satisfy
    $$
         1\leq p<\infty,\ 1\leq q<\infty,\ q(n-2p)<np,
    $$
    and if $n=2$, we moreover require $p>1$ if $q>1$. Then there is a constant $\delta:=\delta(M,g,p,q)>0$ such that
    $$
    \delta\|f-f_{\ave}\|_{L^q(M)}\leq \|\Delta f\|_{L^p(M)},\ \forall f\in C^2(M).$$
      Furthermore, when $p>\frac{n}{2}$, there is a constant $\delta:=\delta(M,g,p)>0$ such that 
    $$\delta\|f-f_{\operatorname{ave}}\|_{L^\infty(M)}\leq \|\Delta f\|_{L^p(M)},\ \forall f\in C^2(M).$$ 
    \end{thm}
     When $p=q=2$, Theorem \ref{thm:laplace compact1} is known, one may see \cite[Lemma 10.4.9]{Nico}.
     
    By the same idea, Theorem \ref{green function uniform}, \ref{thm:laplace boundary} and \ref{thm:laplace compact1} can be generalized to general uniformly elliptic operators of second order.
   We omit the details here.

     We observe that Theorem \ref{thm:laplace boundary} and Corollary \ref{cor:laplace boundary} can be generalized to (quasi)-subharmonic functions (not necessarily continuous) on manifolds. 
      We first recall some notions. 
      
      Assume $(M,g)$ is a Riemannian manifold without boundary, then an upper semi-continuous function $f\colon M\rw [-\infty,\infty)$ is said to be \emph{subharmonic} 
      if $f\not\equiv -\infty$ , and for any open subset $U\subset\subset M$, 
      any function $h\in C^0(\overline U)$ which is harmonic in $U$, we have $f\leq h$ in $U$ if $f\leq h$ on $\partial U$.
      The space of all subharmonic functions on $M$ is denoted by $\operatorname{SH}(M)$. 
      For any $f\in \operatorname{SH}(M)$, one can show that $f$ is locally integrable and $\Delta f\geq 0$ in the sense of distribution.
     
      A function $f\colon M\rw [-\infty,\infty)$ is \emph{quasi-subharmonic} if it can be locally represented by $f=f_1+f_2$ with $f$ subharmonic and $f_2$ in $C^2$. 
      The space of all quasi-subharmonic functions on $M$ is denoted by $\operatorname{QSH}(M)$.
       For any $f\in \operatorname{QSH}(M)$, 
      we use $|\Delta f|$ to denote the variation measure of $\Delta f$.
      We have $|\Delta f|=\Delta f$ if $f$ is subharmonic on $M$.

     Now we can state several integral inequalities about (quasi-)subharmonic functions.
    \begin{thm}\label{thm:subharmonic1}
    Let $(M,g)$ be a Riemannian manifold without boundary, and let $\Omega\subset\subset M$ be an open subset with smooth boundary.
    Assume $p,q$ satisfy 
    $$ 1\leq p<\infty,\ 1\leq q<\frac{np}{n-1},$$
    and if $n=2$, we moreover require $q=1$. Then there is a constant $\delta:=\delta(\Omega,M,g,p,q)>0$ such that 
    $$\delta\|f\|_{L^q(\Omega)}\leq \int_{\Omega}|\Delta f|+\|f\|_{L^p(\partial\Omega)},\ \forall f\in \operatorname{QSH}(M).$$
    \end{thm}
     \begin{thm}\label{thm:subharmonic boundary}
    Let $(M,g)$ be a Riemannian manifold without boundary, and let $\Omega\subset\subset M$ be an open subset with smooth boundary.
    Assume $p,q$ satisfy
    $$
         1\leq p<\infty,\ 1\leq q<\frac{np}{n-1}.
    $$
    Then there is a constant $\delta:=\delta(\Omega,M,g,p,q)>0$ such that
    $$
    \delta\|f\|_{L^q(\Omega)}\leq \|f\|_{L^p(\partial \Omega)}$$
    for all nonnegative $f\in \sh(M)$.
    \end{thm}
    \begin{thm}\label{thm:subharmonic2}
    Let $(M,g)$ be a compact Riemannian manifold without boundary.
    Assume   $p,q$ satisfy 
    $$ 1\leq p<\infty,\ 1\leq q<\frac{np}{n-1},$$
    and if $n=2$, we moreover require $q=1$. Then there is a constant $\delta:=\delta(M,g,p,q)>0$ such that 
    $$\delta\|f-f_{\ave}\|_{L^q(M)}\leq \int_{M}|\Delta f|,\ \forall f\in \operatorname{QSH}(M).$$
    \end{thm}
     When $\Omega\subset\mr^n$ is a bounded open subset, Theorem \ref{thm:subharmonic boundary} is proved by H\"ormander, 
    using a completely different method (\cite[Theorem 3.1]{Hor67}). 
    
    Similar to the proof of Theorem \ref{thm:laplace boundary}, 
     the above three theorems are derived from the following Green-Riesz type formulas (Theorem \ref{thm:green riesz} and \ref{thm:green riesz1}), where for Theorem \ref{thm:subharmonic boundary}, one needs to note the Green function is always non-positive,  so the first term in the right hand side of Equality (\ref{equ:green-riesz}) can be dropped.
    \begin{thm}\label{thm:green riesz}
    Let $(M,g)$ be a Riemannian manifold without boundary, $\Omega\subset\subset M$ be an open subset
    with smooth boundary, and let $G$ be the Green function for the Laplace operator on $\overline\Omega$. Then for any function $f\in \operatorname{QSH}(M)$, we have
    \begin{equation}\label{equ:green-riesz}
    f(x)=\int_{\Omega}G(x,y)\Delta f(y)+\int_{\partial \Omega}f(y)\frac{\partial G(x,y)}{\partial\nve_y}dS(y),\ \forall x\in \Omega.
    \end{equation}
    \end{thm}
    \begin{thm}\label{thm:green riesz1}
    Let $(M,g)$ be a compact Riemannian manifold without boundary, and let $G$ be the Green function of $M$. Then for any function $f\in \operatorname{QSH}(M)$, we have
    \begin{equation*}\label{equ:green-riesz1}
    f(x)=f_{\ave}+\int_{M}G(x,y)\Delta f(y),\ \forall x\in M.
    \end{equation*}
    \end{thm}
     When $M=\mr^n$ and $\Omega$ is a bounded domain, Theorem \ref{thm:green riesz} is well known, see \cite[Proposition 4.22, Chapter I]{Dem} for example. 
     However, we cannot find a precise literature for such a general formula on manifolds.
     The proofs of Theorem \ref{thm:green riesz} and \ref{thm:green riesz1} are given in Section \ref{section:green riesz1}.\\
     \indent One may also consider the corresponding (quasi-)subharmonic functions 
    for a general second order uniformly elliptic operator. Similar results of Theorem \ref{thm:subharmonic1}, \ref{thm:subharmonic boundary}
    and  \ref{thm:subharmonic2} also hold. \\

    Let us turn to some applications of Green functions to complex analysis. Let $(M,\omega)$ be a K\"ahler manifold.
    A function $f\colon M\rw [-\infty,\infty)$ is \emph{quasi-plurisubharmonic} 
    if it can be locally represented as a sum $f_1+f_2$ with $f_1$ plurisubharmonic and $f_2$ is of class $C^2$.
    Since each plurisubharmonic function on an open subset of $M$ is also subharmonic (by using normal coordinates), 
    so Theorem \ref{thm:subharmonic1}, \ref{thm:subharmonic boundary}
    and \ref{thm:subharmonic2} are still valid for quasi-plurisubharmonic functions on $M$.
     
    \indent A Hermitian manifold $(M,h)$ is called balanced if we have 
    $$\Delta_d f=2\Delta_{\partial}f=2\Delta_{\bar\partial }f,\ \forall f\in C^2(M).$$
     Clearly, every K\"ahler manifold is \emph{balanced} by the K\"ahler identities,
    see \cite[Corollary 6.5, Chapter VI]{Dem}, but the converse is not true in general. By a simple observation, we can derive the following integral representation,
    which plays a fundamental role in the proof of the corresponding $\bar\partial$-version of the above inequalities.
    
     \begin{thm}\label{thm:bochner martinella}
    Let $(M,h)$ be a balanced manifold, $\Omega\subset\subset M$ be an open set with  smooth boundary, and let $G$ be the Green 
    function for the Laplace operator on $\overline\Omega$. Then we have 
    $$f(z)=-2\int_{\Omega}\langle \bar\partial f(w),\bar\partial_w G(z,w)\rangle dV(w)+\int_{\partial \Omega}\frac{\partial G(z,w)}{\partial\nve_w}f(w)dS(w)$$
    for all $f\in C^0(\overline\Omega)\cap C^1(\Omega)$ such that $\bar\partial f$ has a continuous continuation to the boundary.
    \end{thm}
    
    Theorem \ref{thm:bochner martinella} gives an integral representation for $C^1$ functions on any balanced manifolds (not necessarily Stein or K\"ahler). 
    It seems that integral representations for functions are mainly considered only on Stein manifolds in the literatures,
    for example, the integral formula discussed in \cite[Theorem 4.3.3]{Henkin},
     whose proof strongly relies on embedding of Stein manifolds into some $\mathbb{C}^n$ and some fundamental results about coherent analytic sheaves on Stein manifolds.
     The approach we adopt here is a direct application of the Green representation formula.
     

    A consequence of Theorem \ref{thm:bochner martinella} is the following $\bar\partial$-version of Sobolev-type inequality, which generalizes \cite[Theorem 1.4]{DJQ}.
    \begin{thm}\label{thm: dbar boundary}
    Let $(M,h)$ be a balanced manifold of complex dimension $n$, $\Omega\subset\subset M$ be an open subset
    with smooth boundary, and let $p,q,r$ satisfy
    $$
         1\leq p,q,r<\infty,\ q(2n-p)<2np,\ q(2n-1)<2nr.
    $$
    There is a constant $\delta:=\delta(\Omega,M,h,p,q,r)>0$ such that
    $$
    \delta\|f\|_{L^q(\Omega)}\leq \|\bar\partial f\|_{L^p(\Omega)}+\|f\|_{L^r(\partial \Omega)}$$
    for all $f\in C^0(\overline\Omega)\cap C^1(\Omega)$ such that $\bar\partial f$ has a continuous continuation to the boundary.
    \end{thm}

We omit the proof of Theorem \ref{thm: dbar boundary} here, 
since, with Theorem \ref{thm:bochner martinella} in hand,  the argument of it is parallel that in the proof of Theorem \ref{thm:graident boundary} and \ref{thm:laplace boundary}. 
Let us record the following special case of Theorem \ref{thm: dbar boundary}, which we think has some independent interests. 
The last author has used its domain version as an important tool to get the curvature strict positivity of direct image bundles (see the proof of \cite[Theorem 1.4]{qin}).
    \begin{cor}\label{cor:dbar poincare}
    Let $(M,h)$ be a balanced manifold of complex dimension $n$,  $\Omega\subset\subset M$ be an open subset
    with smooth boundary, and let $p,q$ satisfy
    $$
         1\leq p<\infty,\ 1\leq q<\infty,\ q(2n-1)<2np.
    $$
     Then there is a constant $\delta:=\delta(M,\Omega,h,p,q)>0$ such that
    $$
    \delta\|f\|_{L^q(\Omega)}\leq \|f\|_{L^p(\partial \Omega)},\ \forall\ f\in C^0(\overline\Omega)\cap \mo(\Omega),$$
    where $\mo(\Omega)$ denotes the space of holomorphic functions on $\Omega$.
    \end{cor}
    \indent Since holomorphic functions are always harmonic, then Corollary \ref{cor:dbar poincare} is a special case of Corollary \ref{cor:laplace boundary}.\\
    \indent  Similar results of Theorem \ref{thm:poincare compact} and \ref{thm:laplace compact1} also hold for the $\bar\partial$-operator.
     In fact, when the considered functions have compact support, we have the following $\bar\partial$-version of the classical Poincar\'e inequality in PDE. 
     \begin{thm}\label{thm:poincare average}
    Let $(M,h)$ be a compact Hermitian manifold of complex dimension $n$, and let $p,q$ satisfy
    $$
         1<p<\infty,\ 1\leq q<\infty,\ q(2n-p)\leq 2np.
    $$
     Then there is a constant $\delta:=\delta(M,h,p,q)>0$ such that
    $$
    \delta\|f-f_{\operatorname{ave}}\|_{L^q(M)}\leq \|\bar\partial f\|_{L^p(M)},\ \forall f\in C^1(M).$$
       Furthermore, when $p>2n$, there is a constant $\delta:=\delta(M,g,p)>0$ such that 
    $$\delta\|f-f_{\operatorname{ave}}\|_{L^\infty(M)}\leq \|\bar\partial f\|_{L^p(M)},\ \forall f\in C^1(M).$$ 
    \end{thm}
    \indent Using partition of unity, Theorem \ref{thm:poincare average} 
     is a consequence of Theorem \ref{thm:poincare compact} 
     and the following lemma, which follows from Proposition 4 of \cite[Section 1, Chapter III]{Stein} and Fubini's Theorem.
    \begin{lem}\label{lem:compare}
    For any $1<p<\infty$, there is a constant $\delta:=\delta(n,p)>0$ such that 
    $$\delta\|\partial f\|_{L^p(\mc^n)}\leq \|\bar \partial f\|_{L^p(\mc^n)},\ \forall f\in C_c^1(\mc^n).$$ 
    \end{lem}
     Using Lemma \ref{lem:compare}, it is possible to prove that if $f\in L_{\loc}^p(\mc^n)$ satisfying $|\bar\partial f|\in L_{\loc}^p(\mc^n)$, then $f$ is in the Sobolev space $W^{1,p}_{\loc}(\mc^n)$ ($1<p<\infty$). When $p=2$, such a fact  is known (\cite[Lemma 2.26]{Ada}). For $p\neq 2$, we cannot find an exact literature for this fact, but the proof 
     of it is essentially the same with  the proof of \cite[Lemma 2.26]{Ada} since one has Lemma \ref{lem:compare}.
     
     Note that we do not put any requirement on the metric $h$ in Theorem \ref{thm:poincare average}, 
      but we require that $p>1$ as we don't know whether Lemma \ref{lem:compare} still holds or not when $p=1$. The condition $p\neq 1$ 
      is very important in the proof of Lemma \ref{lem:compare} since one only has weak $L^1$ estimate for Hardy-Littlewoood maximal functions (see Lemma 
      \ref{lem:maximal function}). Surprisingly, when $p=1$, the conclusion of Theorem \ref{thm:poincare average} also holds if the corresponding manifold is balanced.
     \begin{thm}\label{thm:poincare average1}
    Let $(M,h)$ be a compact balanced manifold of complex dimension $n$, and let 
    $$1\leq q\leq \frac{2n}{2n-1}.$$
     Then there is a constant $\delta:=\delta(M,h,q)>0$ such that
    $$
    \delta\|f-f_{\operatorname{ave}}\|_{L^q(M)}\leq \|\bar\partial f\|_{L^1(M)},\ \forall f\in C^1(M).$$
    \end{thm}
    The main idea for the proof of Theorem \ref{thm:poincare average1} is inspired by the proof of \cite[Theorem 2.11]{Hebey}, but we need to do more efforts, see Section \ref{section:L1} for its proof.\\
    \indent By tracking our proofs, one can see many of the above inequalities can be easily generalized to functions in Sobolev spaces by the density of smooth functions and the continuity of the trace operator.
    
    Finally, using the same method as in \cite{DJQ}, we derive an improved $L^2$-estimate for $\bar\partial$ on a strictly pseudoconvex open subset in a K\"ahler manifold.
     \begin{thm}\label{thm:improved L2 estimate}
     Let $(M,\omega)$ be a K\"ahler manifold,  $\Omega\subset\subset M$ be an open subset with a smooth strictly plurisubharmonic boundary defining function $\rho$, and let $(E,h)$ be a Hermitian holomorphic vector bundle over $M$ such that $A_{E}:=i\Theta^{(E,h)}\wedge \Lambda\geq 0$ in bidegree $(n,1)$. Then there is a constant $\delta:=\delta(M,\omega,\Omega,\rho,E,h)>0$
     such that for any nonzero $\bar\partial$-closed $f\in L^2_{(n,1)}(\Omega,E),$ satisfying 
     $$M_f:=\int_{\Omega}\langle A_E^{-1}f,f\rangle dV<\infty,$$
     there is $u\in L_{(n,0)}^2(\Omega,E)$ such that  $\bar\partial u=f$ and 
     $$\int_{\Omega}|u|^2dV\leq \frac{\|f\|}{\sqrt{\|f\|^2+\delta M_f}}\int_{\Omega}\langle A_E^{-1}f,f\rangle dV.$$
    \end{thm}
    
    Theorem \ref{thm:improved L2 estimate} can be easily generalized to any $(n,q)$-forms and any $q$-pseudoconvex open subsets.

   It is of great interest and importance to consider optimal constants in the estimates presented in all the above theorems and corollaries,
    but we will not discuss this topic in depth in the present work, which will be systematically studied in forthcoming works.\\
    
    \textbf{Acknowledgements.}
    The authors are grateful to Professor Yuan Zhou for discussions of related topics.
    This research is supported by National Key R\&D Program of China (No. 2021YFA1003100),
    NSFC grants (No. 12071310 and No. 12471079), and the Fundamental Research Funds for the Central Universities.
    \section{Preliminaries}
    \subsection{Notations and Conventions}\label{notations}\ \\
    \par In this subsection, we fix some notations, conventions and  collect some knowledge that are needed in our discussions.\\
    \indent Our convention for $\mathbb{N}$ is $\mathbb{N}:=\{0,1,2,\cdots\}.$ For any subset $A$ in a topological space $X$,
    we use $\overline{A}$ to denote its closure in $X$, $A^\circ$ to denote the set of its interior point,
    and set $\partial A:=\overline A\setminus A^\circ.$ If $A,B$ are two subsets of a topological space, then we write $A\subset\subset B$ if $\overline{A}$ is a compact subset of $B.$ \\
    \indent In this article, we only consider oriented smooth manifolds with or without boundary, and when we say a manifold, we always mean it is smooth
    and oriented.  When we consider a smooth manifold $\overline M$ with boundary, then we use $M$ to represent the interior point of $\overline M$ and use $\partial M$ to denote its boundary point, and we always equip $\partial M$ with the induced orientation. Clearly, any open subset of an oriented manifold is oriented and any complex manifold is also oriented.\\ 
    \indent  Let $M$ be a smooth manifold with or without boundary, $k\in \mathbb{N}\cup\{\infty\}$, and let $U$ be an open subset of $M$, then we use  $C^k(U)$ to denote the space of complex valued functions on $U$ which are of class $C^k$, and use $C_c^k(U)$ to denote functions in $C^k(U)$ which have compact support.
    More generally, if $A$ is an arbitrary subset of $M$, then we use $C^k(A)$ to denote the space of complex valued $C^k$-functions on an open neighborhood of $\overline A$.\\
    \indent Let $(\overline M,g)$ be a Riemannian manifold with boundary, then  we use $dV$ to denote the Riemannian volume form on $M$ induced by $g$, and use $dS$ to denote the induced volume form on $\partial M$. For any $A\subset M$ or $A\subset\partial M$, we also write $|A|$ for the volume of $A$. For any $1\leq p\leq\infty$, let $L^p(M)$ (resp. $L^p(\partial M)$) denote the space of complex valued $L^p$-integrable
    functions on $M$ (resp. $\partial M$), and the corresponding norm is denoted by $\|\cdot\|_{L^p(M)}$ (resp. $\|\cdot\|_{L^p(\partial M)}$).
    For any $x,y\in M$, let $d(x,y)$ be the geodesic distance on $M$ induced by $g$. For any $r>0$, set 
    $$B(x,r):=\{y\in M|\ d(x,y)<r\}.$$  
    \indent Let $(M,g)$ be a Riemannian manifold with or without boundary (of dimension $n$), then we use $d$ to represent the usual de Rham operator on $M$.
    Let $\nabla,\Delta$ be the gradient operator and the Laplace operator on $M$, respectively.
    In local coordinate, we may write 
    $$g=\sum_{i,j=1}^n g_{ij}dx_i\otimes dx_j,$$
    then we use $\det g:=\det((g_{ij})_{1\leq i,j\leq n})$ to denote  the determinant of $g$, and use $(g^{ij})_{1\leq i,j\leq n}$ to represent the inverse matrix of $(g_{ij})_{1\leq i,j\leq n}$.  Clearly, we have 
    $$\nabla f=\sum_{i,j}g^{ij}\frac{\partial f}{\partial x_i}\frac{\partial}{\partial x_j},\ \Delta f=\frac{1}{\sqrt{\det g}}\sum_{i,j=1}^n\frac{\partial}{\partial x_i}\left(\sqrt{\det g}g^{ij}\frac{\partial f}{\partial x_j}\right).$$
    For any $k\in\mathbb{N}$, let $\phi,\psi$ be two complex valued measurable $k$-forms on $M$, which have local expressions
    $$\phi=\frac{1}{k!}\sum_{1\leq i_1,\cdots,i_k\leq n}\phi_{i_1\cdots i_k}dx_{i_1}\wedge\cdots\wedge dx_{i_k},$$
    $$\psi=\frac{1}{k!}\sum_{1\leq j_1,\cdots,j_k\leq n}\psi_{j_1\cdots j_k}dx_{j_1}\wedge\cdots\wedge dx_{j_k},$$
    then we define 
   $$\langle \phi,\psi\rangle:=\frac{2^k}{k!}g^{i_1j_1}\cdots g^{i_k j_k}\phi_{i_1\cdots i_k}\overline{\psi_{j_1\cdots j_k}},$$
   and we also set 
   $$|\phi|^2:=\langle \phi,\phi\rangle.$$
   This definition is independent of the choice of local coordinates. For any $f\in C^1(M)$, we also set $|\nabla f|:=|df|.$ For any $1\leq p<\infty$, set
   $$\|\phi\|_{L^p(M)}:=\left(\int_{M}|\phi|^pdV\right)^{\frac{1}{p}},$$
   then we say $\phi$ is $L^p$-integrable if $\|\phi\|_{L^p(M)}<\infty$.
   If $\phi$ and $\psi$ are both $L^2$-integrable, then  we can define the global inner product of $\phi,\psi$ via 
    $$(\phi,\psi):=\int_M\langle \phi,\psi\rangle dV.$$  
     Let $d^*$ be the formal adjoint of $d$, then the Hodge-Laplace operator of $d$ is given by 
     $$\Delta_d:=dd^*+d^*d.$$
    \indent Let $(M,g)$ be a Riemannian manifold with or without boundary, the Hodge star operator $*$ is defined by the following formula
    $$\phi\wedge *\bar{\psi}=\langle \phi,\psi\rangle dV,$$
    where $\phi,\psi$ are any complex valued measurable $k$-forms for any $k\in\mathbb{N}$. It is clear that $*$ is an isometry operator.\\
    \indent  Let $M$ be a complex manifold which is equipped a smooth Riemannian metric $g$. Choose local holomorphic coordinates $z_j:=x_j+iy_j$, then we have 
    $$dz_j=dx_j+\sqrt{-1}dy_j,\ d\bar{z}_j=dx_j-\sqrt{-1}dy_j,$$
    and so the above definitions give a Hermitian metric on the space complex valued measurable $(p,q)$-forms (via complexification),
    which is still denoted by $\langle\cdot,\cdot\rangle.$ In particular, this gives a Hermitian metric on $M$. We say $(M,h)$ is a Hermitian manifold and $(M,g)$ is the underlying Riemannian manifold.\\
    \indent Let $M$ be a Hermitian manifold, and let $(E,h)$ be a Hermitian holomorphic vector bundle over $M$.
    We use $\bar\partial$ to denote the dbar operator on $M$. Set $\partial:=d-\bar\partial$, then we use $\partial^*$ and $\bar\partial^*$ to denote the formal adjoint of $\partial$ and $\bar\partial^*$, respectively. The Hodge-Laplace operator   of $\partial$ and $\bar\partial$ are given by 
     $$\Delta_\partial:=\partial\partial^*+\partial^*\partial,\ \Delta_{\bar\partial}:=\bar\partial\bar\partial^*+\bar\partial^*\bar\partial.$$
     For any $p,q\in \mathbb{N}$, let $\Lambda^{p,q}T^*M$ be the bundle of smooth $(p,q)$-forms on $M$, and let $\Lambda^{p,q}T^*M\otimes E$  denote the bundle of smooth $E$-valued $(p,q)$-forms on $M$. It is easy to see that the operator $\bar\partial$ can be naturally extended to smooth $E$-valued $(p,q)$-forms on $M$, and the 
      corresponding  extension operator will be denoted by $\bar\partial^E$. Let $D^E$ be the Chern connection on $E$,
      $\Theta^{(E,h)}$ be the Chern curvature tensor of $(E,h).$ Let $\partial^E$ be the $(1,0)$-part of $D^E$, then we have $D^E=\partial^E+\bar\partial^E$. 
       We use $\partial^{E,*}$ and $\bar\partial^{E,*}$ to denote the formal adjoint of $\partial^E$ and $\bar\partial^E$, respectively. There is a natural inner product on $\Lambda^{p,q}T^*M\otimes E$ which is given by 
    $$\langle \alpha_1\otimes e_1,\alpha_2\otimes e_2\rangle=\langle \alpha_1,\alpha_2\rangle h(e_1,e_2\rangle,$$
    where $\alpha_1,\alpha_2$ are complex valued measurable $(p,q)$-forms and $e_1,e_2$ are measurable sections of $E$. For any measurable sections $\alpha,\beta$ of $\Lambda^{p,q}T^*M\otimes E$, we set 
    $$(\alpha,\beta):=\int_{M}\langle \alpha,\beta\rangle dV,\ \|\alpha\|:=\left(\int_{M}|\alpha|^2dV\right)^{\frac{1}{2}}.$$  
    Let $L_{(p,q)}^2(M,E)$ denote the space the complex valued measurable sections $\alpha$ of $\Lambda^{p,q}T^*M\otimes E$ such that 
    $\|\alpha\|<\infty.$ Let $A$ be a subset of $M$ and let $k\in\mathbb{N}\cup\{\infty\}$, we use 
    $C^k_{(p,q)}(\overline A,E)$ to denote elements of $L_{(p,q)}^2(M,E)$ which are of class $C^k$ in an open neighborhood of 
    $\overline A$. Now we moreover assume $M$ is a K\"ahler manifold, and let $\omega$ be its K\"ahler form.
    Let $\alpha$ (resp. $\beta$) be any  measurable $E$-valued $(p,q)$-form (resp. $(p+1,q+1)$-form),
    then the Lefschetz operator $L$ is defined by $L\alpha=\omega\wedge\alpha$, and its adjoint $\Lambda$ is given by the following formula 
    $$\langle L\alpha,\beta\rangle=\langle \alpha,\Lambda\beta\rangle.$$ 
    \indent  Let $\Omega\subset\mc^n$ be an open subset, then an upper semi-continuous function $f\colon \Omega\rw [-\infty,\infty)$
        is plurisubharmonic if its restriction to every complex line in $\Omega$ is subharmonic, i.e.,satisfying the submean value inequality.
    A $C^2$ function $f$ on an open subset $\Omega\subset\mc^n$ is strictly plurisubharmonic if its complex Hessian 
    $\left(\frac{\partial^2 f}{\partial z_j\partial\bar{z}_k}(z)\right)_{1\leq j,k\leq n}$ is positive definite for any $z\in\Omega$.
    Since being (strictly) plurisubharmonic is invariant under biholomorphic coordinate change, then we say a $C^2$-function $f$ on a complex manifold $M$
    is (strictly) plurisubharmonic if the restriction to every coordinate open subset is (strictly) plurisubharmonic. 
    Moreover, an open subset  $\Omega\subset M$ with smooth boundary is strictly pseudoconvex if there is a strictly pluisubharmonic function $\rho$ which is defined 
    on an open neighborhood of $\overline\Omega$ such that $\rho$ is the boundary defining function for $\Omega$. For brevity, 
    we say $\Omega$ is strictly pseudoconvex if it has a smooth strictly plurisubharmonic boundary defining function.
    \subsection{Green functions on manifolds}\ \\
    \par In this subsection, we recall the definitions and some properties of Green functions.
    Firstly, we give the definition of Green functions. We use $\Delta^{\distr}$ to denote the distributional Laplacian operator on a Riemannian manifold. 
    \begin{defn}
    Let $(M,g)$ be a compact Riemannian manifold without boundary, then a Green function for $M$ is 
    a function $G\colon M\times M\rw \mr\cup\{-\infty\}$ which satisfies the following properties:
    \begin{itemize}
      \item[(i)] $G\in C^\infty((M\times M)\setminus\diag),$ where $\diag:=\{(x,x)|\ x\in M\}.$ 
      \item[(ii)] For any $x\in M$, $G(x,\cdot)$ is locally integrable near $x$.
      \item[(iii)] For any $x\in M$, the distributional Laplacian of $G$ satisfies 
      $$\Delta_y^{\distr}G(x,y)=\delta_x(y)-\frac{1}{|M|},$$
      where $\delta_x(\cdot)$ denotes the Dirac distribution which is concentrated at the point $x$.
    \end{itemize}
    \end{defn}
    \begin{defn}
    Let $(\overline M,g)$ be a compact Riemannian manifold with boundary, then a Green function for $\overline M$ is 
    a function $G\colon \overline M\times \overline M\rw \mr\cup\{-\infty\}$ which satisfies the following properties:
    \begin{itemize}
      \item[(i)] $G\in C^\infty((\overline M\times \overline M)\setminus\diag),$ where $\diag:=\{(x,x)|\ x\in \overline M\}.$ 
      \item[(ii)] For any $x\in M$, $G(x,\cdot)$ is locally integrable near $x$.
      \item[(iii)] For any $x\in M$, the distributional Laplacian of $G$ satisfies 
      $$\Delta_y^{\distr}G(x,y)=\delta_x(y).$$
     \item[(iv)] For any $x\in M$, $G(x,\cdot)|_{\partial M}=G(\cdot,x)|_{\partial M}=0$.
    \end{itemize}
    \end{defn}
    The above definitions can be easily modified to give the definitions of Green functions for a general elliptic operator on a manifold.\\
    \indent Now we state some properties of Green functions that will be used later.
    \begin{lem}\cite[Theorem 4.13]{Aubin}\label{green function:compact}
    Let $(M,g)$ be a compact Riemannian manifold without boundary (of dimension $n\geq 2$), then there is a Green function $G$ on $M$ which has 
    the following extra properties:
     \begin{itemize}
     \item[(i)] $G(x,y)=G(y,x)$ for all $x,y\in M,\ x\neq y$.
      \item[(ii)] For any $f\in C^2(M)$ and any $x\in M$, we have 
      $$f(x)=f_{\ave}+\int_{M}G(x,y)\Delta f(y)dV(y).$$
      \item[(iii)] If $n\geq 3$,  there is a constant $C:=C(M,g)>0$ such that 
     $$-Cd(x,y)^{2-n}\leq G(x,y)\leq C,\ |\nabla_yG(x,y)|\leq Cd(x,y)^{1-n}$$
     for all $x,y\in M,\ x\neq y.$
     \item[(iv)] If $n=2$,  there is a constant $C:=C(M,g)>0$ such that 
     $$-C(1+|\ln |d(x,y)||)\leq G(x,y)\leq C,\ |\nabla_yG(x,y)|\leq Cd(x,y)^{-1}$$
     for all $x,y\in M,\ x\neq y.$
    \end{itemize}
    \end{lem}
    \begin{lem}\cite[Theorem 4.17]{Aubin}\label{green function:noncompact}
    Let $(\overline M,g)$ be a compact Riemannian manifold with boundary (of dimension $n\geq 2$), then there is a unique Green function $G$ on $M$ which has 
    the following extra properties:
     \begin{itemize}
      \item[(i)] For all $f\in C^2(\overline M),\ x\in M$, we have 
     $$f(x)=\int_{M}G(x,y)\Delta f(y)dV(y)+\int_{\partial M}\frac{\partial G(x,y)}{\partial\nve_y}f(y)dS(y),$$
     where $\nve$ is the unit outward normal vector of $\partial M$.
     \item[(ii)] $G(x,y)=G(y,x)\leq 0$ for all $(x,y)\in (\overline M\times\overline M)\setminus\diag$.
    \end{itemize}
    \end{lem}
    \subsection{Some useful lemmas}\ \\
    \par In this subsection, we collect some useful lemmas that will be used later.
    \begin{lem}\cite[Theorem 4.7]{Aubin}\label{Aubin:Theorem 4.7}
    Let $(M,g)$ be a compact Riemannian manifold without boundary, and let $f\in C^\infty(M)$ such that $\int_{M}fdV=0$,
    then we may find $u\in C^\infty(M)$ such that $\Delta u=f$ and $\int_{M}udV=0.$
    \end{lem}
     The following lemma is essentially proved in the proof of Part (b) of \cite[Theorem 4.13]{Aubin}.
    \begin{lem}\label{Aubin:Theorem 4.7b}
    Let $(M,g)$ be a compact Riemannian manifold without boundary, then there is a constant $C:=C(M,g)>0$ such that 
    $$\sup_{M}|u|\leq C\sup_M|f|,$$
    for all  $u,f\in C^\infty(M)$ satisfying 
    $$\Delta u=f,\ \int_M udV=\int_{M}fdV=0.$$ 
    \end{lem}
    By \cite[Theorem 6.24]{Gilbarg},\ \cite[Theorem 4.8]{Aubin} and Weyl's lemma (see \cite[Theorem 9.19]{Gilbarg}), we may easily get 
    \begin{lem}\cite[Theorem 4.8]{Aubin}\label{Aubin:Theorem 4.8}
    Let $(\overline M,g)$ be a compact Riemannian manifold with boundary, then for any $f\in C^0(\partial M)$,
    we may find $u\in C^0(\overline M)\cap C^\infty(M)$ such that 
    $$\Delta u=0\text{ in }M,\ u|_{\partial M}=f.$$
    Moreover, if $f$ is smooth near the boundary, then $u\in C^\infty(\overline M)$.
    \end{lem}
    The following Cheng-Yau's gradient estimate is well known.
    \begin{lem}[{\cite[Theorem 6.1]{Li}}]\label{gradient estimate 1}
    Let $(M,g)$ be a compact Riemannian manifold without boundary (of dimension $\geq 2$), then 
    there is a constant $C:=C(M,g)>0$ such that for any $x_0\in M$, any $r>0$, and any positive harmonic function $u$ on $B(x_0,r)$, we have 
    $$\sup_{B\left(x_0,\frac{r}{2}\right)}|\nabla u(x)|\leq \frac{C}{r}\sup_{B\left(x_0,\frac{r}{2}\right)}u.$$
    \end{lem}
    \section{Uniform estimates of Green functions on $\mr^n$}
    In this section, we will consider the uniform estimates of Green functions for constant coefficient elliptic operators,
    which will be used to prove Lemma \ref{gradient estimate 2}. By a coordinate change, we only need to consider the case of Laplace operator.
    Indeed, we have 
    \begin{thm}\label{estimate:euclidean}
    Let $\Omega\subset\mr^n(n\geq 2)$ be a bounded open subset with $C^2$-boundary, then there is a function $G$ on 
    $\Omega\times \Omega$, which satisfies the following properties:
    \begin{itemize}
      \item[(i)] For all $f\in C^2(\overline \Omega),\ x\in \Omega$, we have 
     $$f(x)=\int_{\Omega}G(x,y)\Delta f(y)dV(y)+\int_{\partial \Omega}\frac{\partial G(x,y)}{\partial\nve_y}f(y)dS(y).$$
      \item[(ii)] $G\in C^2(\Omega\times \Omega)\setminus\diag)$, where $\diag:=\{(x,x)|\ x\in \Omega\}.$ Moreover, we know $G\leq 0$.
      \item[(iii)] For any $(x,y)\in (\Omega\times\Omega)\setminus\diag$, we have 
                  $$ |G(x,y)|\leq \left\{\begin{array}{ll}
                  |x-y|^{2-n} & \text{ if }n\geq 3,\\
                  |\ln|x-y||+|\ln|\diam(\Omega)||  & \text{ if }n=2,
                  \end{array}\right.$$
                  where $\diam(\Omega)$ is the diameter of $\Omega$.
      \item[(iv)] There is a constant $C:=C(\Omega)>0$ such that for all $(x,y)\in (\Omega\times\Omega)\setminus\diag$,
                  we have 
                  $$|G(x,y)|\leq C\cdot\left\{\begin{array}{ll}
                  |x-y|^{1-n}\delta(x) & \text{ if }n\geq 3,\\
                  (1+|\ln|x-y||)\frac{\delta(x)}{|x-y|}  & \text{ if }n=2.
                  \end{array}\right.$$
                  where $\delta(x):=\inf_{z\in\partial\Omega}|x-z|.$
      \item[(v)] There is a constant $C:=C(\Omega)>0$ such that for all $(x,y)\in (\Omega\times\Omega)\setminus\diag$,
                  we have 
                  $$|\nabla_y G(x,y)|\leq C\cdot\left\{\begin{array}{ll}
                  |x-y|^{1-n}\min\left\{1,\frac{\delta(x)}{|x-y|}\right\} & \text{ if }n\geq 3,\\
                  \frac{1+|\ln|x-y||}{|x-y|}\cdot\min\left\{1,\frac{\delta(x)}{|x-y|}\right\}  & \text{ if }n=2.
                  \end{array}\right.$$
      \item[(vi)] There is a constant $C:=C(\Omega)>0$ such that for all $(x,y)\in (\Omega\times\Omega)\setminus\diag$,
                  we have 
                  $$|\nabla_x\nabla_y G(x,y)|\leq C\cdot\left\{\begin{array}{ll}
                  |x-y|^{-n} & \text{ if }n\geq 3,\\
                  \frac{1+|\ln|x-y||}{|x-y|^2}  & \text{ if }n=2.
                  \end{array}\right.$$                 
    \end{itemize}
    \end{thm}
    \begin{proof}
    Existence of $G$, and properties (i),(ii),(iii) are well known. (iv) and (v) and (vi) are also known for the case $n\geq 3$, see 
    \cite[Theorem 3.3]{Widman}, but we contain a complete proof here since this method will be used to get the uniform estimate of Green functions for the Laplace operator
    on manifolds. We may assume $n\geq 3$, the case $n=2$ can be handled similarly.\\
    (iv) Since $\Omega$ is bounded and has $C^2$-boundary, then it satisfies the uniform exterior sphere condition,
        i.e., there is $r_0>0$ satisfying for any $z\in\partial\Omega$ and any $0<r<r_0$, there is an open ball $B^z(r)$ with radius $r$ such that 
        $$B^z(r)\subset\mr^n\setminus\overline{\Omega},\ \overline{B^z(r)}\cap \partial\Omega=\{z\}.$$ 
        Fix such an $r_0$ and $(x,y)\in(\Omega\times\Omega)\setminus\diag$, then we consider two cases.\\
    {\bf Case 1:} $\delta(x)<\min\left\{\frac{|x-y|}{8},r_0\right\}$.\\
    \indent Set $r:=\min\left\{\frac{|x-y|}{8},r_0\right\}$. Choose $z_x\in\partial\Omega$ and choose $x^*\in\mr^n\setminus\overline{\Omega}$
     such that 
     $$|x-z_x|=\delta(x),\ B^{z_x}(r)=B(x^*,r).$$
     Set 
    $$u(z):=
      \frac{2^n}{4-2^n}\left[\left(\frac{r}{|z-x^*|}\right)^{n-2}-1\right],\ \forall\ z\in \mr^n\setminus\{x^*\},$$
    then we know 
    $$
    \left\{\begin{array}{ll}
      \Delta u=0&\text{ in }\mr^n\setminus\{x^*\},\\
       u=0 & \text{ on }\partial B(x^*,r),\\
       u=1 & \text{ on }\partial B(x^*,2r).
    \end{array}\right.
    $$ Moreover, it is clear that 
    $$\sup_{z\in\mr^n\setminus B(x^*,r)}|\nabla u(z)|\leq 
      \frac{2^n(n-2)}{(2^n-4)r}\leq \frac{n}{r}.$$
    As $z_x\in\partial B(x^*,r)$, then by the mean value theorem, we know  
    $$u(x)=|u(x)-u(z_x)|\leq \frac{n|x-z_x|}{r}=\frac{n\delta(x)}{r}.$$
    For any $z\in\partial B(x^*,2r)\cap\Omega$, we have 
    $$|z_x-z|\leq |z_x-x^*|+|x^*-z|=3r,\ |x-z|\leq |x-z_x|+|z_x-z|\leq 4r,$$
    $$|z-y|\geq |x-y|-|x-z|\geq |x-y|-4r=\frac{|x-y|}{2},$$
    so by (iii), we know 
    $$
    |G(z,y)|\leq |z-y|^{2-n}\leq 2^{n-2}|x-y|^{2-n}\leq 2^{n}|x-y|^{2-n}u(z).
    $$
    Note that for any $z\in \partial\Omega$, we have $G(z,y)=0$. Applying the maximum principle to the set 
    $\Omega\cap (B(x^*,2r)\setminus \overline{B(x^*,r)})\ni x$, we know 
    $$|G(x,y)|\leq 2^{n}|x-y|^{2-n}u(x)\leq n2^n|x-y|^{2-n}\frac{\delta(x)}{r},$$
    then we have
    \begin{equation}\label{gradient:eequ1}
     |G(x,y)|\leq n2^n\left(8+\frac{\diam(\Omega)}{r_0}\right)|x-y|^{1-n}\delta(x). 
    \end{equation}
    {\bf Case 2:} $\delta(x)\geq \min\left\{r_0,\frac{|x-y|}{8}\right\}$.\\
    \indent In this case, we have 
    $$\frac{|x-y|}{\delta(x)}\leq 8+\frac{\diam(\Omega)}{r_0},$$
    so by (iii) we know
    \begin{equation}\label{gradient:eequ2}
    |G(x,y)|\leq |x-y|^{2-n}\leq \left(8+\frac{\diam(\Omega)}{r_0}\right)|x-y|^{1-n}\delta(x).
    \end{equation}
    \indent  Combining {\bf Case 1} and {\bf Case 2}, especially Inequalities (\ref{gradient:eequ1}) and (\ref{gradient:eequ2}),
    we know (iv) holds with 
    $$C:=n2^n\left(8+\frac{\diam(\Omega)}{r_0}\right).$$
    (v) We only prove 
    $$|\nabla_y G(x,y)|\leq C|x-y|^{1-n}$$
    for some constant $C>0.$ Since the other inequality can be derived as (iv) was implied by (iii)
    if we note that for fixed $y\in\Omega$, $\nabla_y G(\cdot,y)$ is a solution of $\Delta u=0$ in $\Omega\setminus\{y\}$. Fix $x,y\in\Omega,\ x\neq y$, we also consider two cases.\\
    {\bf Case 1:} $\delta(y)\leq |x-y|$.\\
    \indent In this case, we know $G(x,\cdot)$ is a harmonic function in $B\left(y,\frac{1}{2}\delta(y)\right)$. By the following well known Lemma 
    \ref{gradient estimate 3}, we know 
    $$|\nabla_y G(x,y)|\leq \frac{n}{\frac{1}{2}\delta(y)}\sup_{B\left(y,\frac{1}{2}\delta(y)\right)}|G(x,\cdot)|.$$
    For any $z\in B\left(y,\frac{1}{2}\delta(y)\right)$, we have 
    $$|x-z|\geq |x-y|-|y-z|\geq |x-y|-\frac{1}{2}\delta(y)\geq \frac{1}{2}|x-y|,$$
    $$\delta(z)\leq \delta(y)+|y-z|\leq 2\delta(y).$$ 
    By (iv), we know 
    \begin{align*}
    |G(x,z)|&\leq n2^n\left(8+\frac{\diam(\Omega)}{r_0}\right)|x-z|^{1-n}\delta(z)\\
         &\leq n2^{2n-1}\left(8+\frac{\diam(\Omega)}{r_0}\right)|x-y|^{1-n}\delta(y).
    \end{align*}
    Thus, we have 
    \begin{equation}\label{gradient:eequ3}
    |\nabla_y G(x,y)|\leq n4^{n}\left(8+\frac{\diam(\Omega)}{r_0}\right)|x-y|^{1-n}.
    \end{equation}
    {\bf Case 2:} $\delta(y)>|x-y|$. \\
    \indent In this case, we know $G(x,\cdot)$ is a harmonic function in $B\left(y,\frac{1}{2}|x-y|\right)$. By Lemma \ref{gradient estimate 3}, we get 
    $$|\nabla_y G(x,y)|\leq \frac{n}{\frac{1}{2}|x-y|}\sup_{B\left(y,\frac{1}{2}|x-y|\right)}|G(x,\cdot)|.$$
    For any $z\in B\left(y,\frac{1}{2}|x-y|\right)$, we have 
    $$|x-z|\geq |x-y|-|y-z|\geq \frac{1}{2}|x-y|,$$
    By (iii), we have 
    $$|G(x,z)|\leq |x-z|^{2-n}\leq 2^{n-2}|x-y|^{2-n}.$$
    Thus, we get 
    \begin{equation}\label{gradient:eequ4}
    |\nabla_y G(x,y)|\leq n2^n|x-y|^{1-n}.
    \end{equation}
    \indent Combining Inequalities (\ref{gradient:eequ3}) and (\ref{gradient:eequ4}), we know the  
    $$|\nabla_y G(x,y)|\leq n4^{n}\left(8+\frac{\diam(\Omega)}{r_0}\right)|x-y|^{1-n}.$$
    (iv) Note that for fixed $y\in\Omega$, $\nabla_y G(\cdot,y)$ is a solution of $\Delta u=0$ in $\Omega\setminus\{y\}$.
     Now (iv) follows from (v) as (v) was implied by (iii) and (iv). 
    \end{proof}
    The following lemma is well known.
    \begin{lem}\label{gradient estimate 3}\cite[Theorem 2.10]{Gilbarg}
    Let $\Omega\subset\mr^n$, and let $u\in C^2(\Omega)$ be a harmonic function, then for any $x\in\Omega$, we have 
    $$|\nabla u(x)|\leq \frac{n}{\delta(x)}\sup_{\Omega}|u|.$$
    \end{lem}
    Using a coordinate change, we may get a direct consequence of Theorem \ref{estimate:euclidean}, which will be used in the proof of Lemma \ref{gradient estimate 2}. 
    This also follows from Theorem 3.3 of \cite{Widman}.
    \begin{cor}\label{cor:euclidean green}
    Let $\Omega\subset\mr^n(n\geq 2)$ be a bounded open subset with $C^2$-boundary, let $L:=\sum_{i,j=1}^na_{ij}\frac{\partial^2}{\partial x_i\partial x_j}$
    be a constant coefficient elliptic operator, i.e. $(a_{ij})_{1\leq i,j\leq n}$ is a symmetric positive definite matrix, and let $G$ be the corresponding Green function 
    for $L$. Then there is a constant $C:=C(\Omega,L)>0$ such that
    $$|G(x,y)|\leq C\cdot\left\{\begin{array}{ll}
                  |x-y|^{2-n} & \text{ if }n\geq 3,\\
                  1+|\ln|x-y||  & \text{ if }n=2,
                  \end{array}\right.$$
    $$ |\nabla_y G(x,y)|\leq C\cdot\left\{\begin{array}{ll}
                  |x-y|^{1-n} & \text{ if }n\geq 3,\\
                  \frac{1+|\ln|x-y||}{|x-y|} & \text{ if }n=2,
                  \end{array}\right.$$
    $$|\nabla_x\nabla_y G(x,y)|\leq C\cdot\left\{\begin{array}{ll}
                  |x-y|^{-n} & \text{ if }n\geq 3,\\
                  \frac{1+|\ln|x-y||}{|x-y|^2}  & \text{ if }n=2.
                  \end{array}\right.$$ 
    \end{cor}
    \section{Gradient estimates of harmonic functions}
    In this section, using the method of freezing coefficients, we will prove a key lemma, which will play a crucial role in the proof of Theorem 
    \ref{thm:Green function}.
    \begin{lem}\label{gradient estimate 2}
    Let $(M,g)$ be a compact Riemannian manifold without boundary (of dimension $\geq 2$), $r_0>0$ be the injective radius of $M$,
    and let $0<4r<r_0.$ Fix $x_0\in M$, and set $E:=B(x_0,2r)\setminus\overline{B(x_0,r)}$. Suppose $u\in C^2(\overline{E})$ satisfies 
    $$
    \left\{\begin{array}{ll}
      \Delta u=0&\text{ in }E,\\
       u=0 & \text{ on }\partial B(x_0,r),\\
       u=1 & \text{ on }\partial B(x_0,2r),
    \end{array}\right.
    $$
    then there is a constant $C:=C(M,g)>0$ such that 
    $$\sup_{E}|\nabla u|\leq \frac{C}{r}.$$ 
    \end{lem}
    \begin{proof}
    The proof is inspired by the proof of \cite[Lemma 3.2]{Widman}. We may assume $\dim(M)\geq 3$, as the case $\dim(M)=2$ is similar
    if we note that there is a constant $C:=C(\diam(M))>0$ such that 
    $$1+|\ln t|\leq \frac{C}{t^{\frac{1}{2}}}\text{ for all } 0<t\leq \diam(M).$$
     By the maximum principle, we know $\sup_{x\in E}|u(x)|\leq 1$.
    Choose a local coordinate $(x_1,\cdots,x_n)$ on $B(x_0,4r),$  and set 
    $$D_i:=\frac{\partial}{\partial x_i},\ D_{ij}:=\frac{\partial^2}{\partial x_i\partial x_j},\
     g=\sum_{i,j} g_{ij} dx_i\otimes dx_j.$$ 
    If there may be some confusions, we use $D_{x_i}$ to replace $D_i$. Let $N:=\sup_{E}|\nabla u|$, and choose $y_0\in E$ such that $|\nabla u(y_0)|\geq \frac{N}{2}.$
    Choose a cut-off function $\chi\in C_c^\infty(\mr)$ such that
    $$0\leq \chi\leq 1,\ \chi|_{B(0,\frac{1}{2})}=1,\ \supp(\chi)\subset B(0,1),\ |\chi'|\leq 2,\ |\chi''|\leq 2.$$ 
    Set 
    $$\eta(x):=\chi\left(\frac{d(y_0,x)}{rs}\right),\ \forall x\in M,$$
    where $0<s<1$ to be determined later. Note that $\overline{B(y_0,r)}\subset B(x_0,3r)$. It is clear that 
    $$\sup_{B(y_0,r)}|\nabla d(y_0,\cdot)|\leq 1.$$
    Using the Hessian comparison theorem (e.g. see the first answer of \url{https://math.stackexchange.com/questions/1499968/}), there is a constant $C_1:=C_1(M,g)>0$ (in the following, when we say a number is a constant,
    we mean it only depends on $M,g$) such that 
    $$\sum_{i,j}|D_{ij}d(y_0,x)|\leq C\left(\frac{1}{d(y_0,x)}+1\right),\ \forall x\in B(y_0,r),$$ 
    so there is a constant $C_2>0$ such that 
    $$\sup_{M}\sum_i|D_i\eta|\leq \frac{C_2}{rs},\ \sup_{M}\sum_{i,j}|D_{ij}\eta|\leq \frac{C_2}{(rs)^2}.$$
    \indent We may assume $u\eta=0,\ uD_i\eta=0$ near $\partial E$ for any $i$. If $d(y_0,x_0)\leq \frac{3r}{2}$, this is obvious.
    If $d(y_0,x_0)>\frac{3r}{2}$, we may replace $u$ by $1-u$. This is a solution of the equation $\Delta u=0$ but with boundary 
    values $1$ on $\partial B(x_0,r)$ and $0$ on $\partial B(x_0,2r)$, and we again have 
    $$\sup_{E}|\nabla(1-u)|=N,\ |\nabla(1-u)(y_0)|\geq \frac{N}{2}.$$
    \indent Let $G$ be the Green function in $E$ corresponding to the operator 
    $$L:=\sum_{i,j}g^{ij}(y_0)D_{ij},$$
    so we have 
    $$u(x)\eta(x)=\int_{E}G(x,y)L(u\eta)(y)dV(y),\ \forall\ x\in E.$$
    Differentiating the above equality, for any $k$ and any $x\in E$,  we have 
    \begin{align}\label{gradient:equ1}
    &\eta(x)D_ku(x)+u(x)D_k\eta(x)\nonumber\\
    =&\sum_{i,j}\int_{E}D_{x_k}G(x,y)g^{ij}(y_0)D_{ij}(u\eta)(y)dV(y)\\
    =&\sum_{i,j}\int_{E}D_{x_k}G(x,y)D_{i}\left(\left(g^{ij}(y_0)-g^{ij}(y)\right)D_j(u\eta)\right)(y)dV(y)\nonumber\\
    &+\sum_{i,j}\int_{E}D_{x_k}G(x,y)D_{i}\left(g^{ij}(y)D_j(u\eta)\right)(y)dV(y).\nonumber
    \end{align}
    \indent By Corollary \ref{cor:euclidean green}, we may choose a constant $C_3>0$ such that   
    $$|G(x,y)|\leq C_3d(x,y)^{2-n},\ \sum_k|D_{x_k}G(x,y)|\leq C_3d(x,y)^{1-n},$$
    $$\sum_{k,l}|D_{y_l}D_{x_k}G(x,y)|\leq C_3d(x,y)^{-n},\ \forall x,y\in\overline E,\ x\neq y.$$
    Set 
    $$b_i:=\sum_j g^{ij}D_j\sqrt{\det g}.$$
    There also exists a constant $C_4>0$ such that 
    $$\sum_{i,j}|g^{ij}(y_0)-g^{ij}(x)|\leq C_4d(y_0,x),\ \forall x\in\overline{E}\setminus\{y_0\};$$
    $$\frac{1}{C_4^2}\delta^{ij}\leq g^{ij}(x)\leq C_4^2\delta^{ij},\ \sum_i|b_i(x)|\leq C,\sum_{i,j}|g^{ij}(x)|\leq C_4,$$
     $$\sqrt{\det g(x)}\leq C_4,\ \sum_{i,j}|D_i(g^{ij}(x))|\leq C,\ \sum_i D_i\sqrt{\det g(x)}\leq C,\ \forall x\in E.$$
    Since 
    $$|\nabla u|^2=g^{ij}\frac{\partial u}{\partial x_i}\frac{\partial u}{\partial x_j},$$
    then we have 
    $$\sum_j\sup_{E}|D_ju|\leq C_4N.$$
    Set $C:=\max\{C_1,C_2,C_3,C_4\}.$\\
    \indent Since 
    $$\int_{\partial  B(y_0,\epsilon)}d(y_0,y)^{2-n}dS(y)\rw 0\text{ as }\epsilon\rw 0+,$$
    then we know 
    $$\int_{\partial B(y_0,\epsilon)}\sqrt{\det g(y)}D_{x_k}G(x,y)|_{x=y_0}(g^{ij}(y_0)-g^{ij}(y))D_j(u\eta)(y)dS(y)\rw 0$$
    whenever $\epsilon\rw 0+.$ Let $d\lambda$ be the usual Lebesgue measure on $B(x_0,4r)$ (which has been identified with its image by the above chosen coordinate).
    Thus, using integration by parts and the fact $D_j(u\eta)=0$ on $\partial E$, for any $k$ and any $x\in E$, we get 
    \begin{align*}
    &\sum_{i,j}\left.\int_{E}D_{x_k}G(x,y)D_{i}((g^{ij}(y_0)-g^{ij}(y))D_j(u\eta))(y)dV(y)\right|_{x=y_0}\\
    =&-\sum_{i,j}\left. \int_{E}D_{y_i}(\sqrt{\det g(y)}D_{x_k}G(x,y))(g^{ij}(y_0)-g^{ij}(y))D_j(u\eta)(y)d\lambda(y)\right|_{x=y_0}\\
    =&-\sum_{i,j} \int_{E}D_{y_i}(\sqrt{\det g(y)}D_{x_k}G(x,y)|_{x=y_0})(g^{ij}(y_0)-g^{ij}(y))D_j(u\eta)(y)d\lambda(y).
    \end{align*}
    Note that $\supp(\eta)\subset B(y_0,rs)$, then there is a constant $C_5>0$ such that  
    \begin{align}\label{gradient:equ2}
    &\left|\sum_{i,j}\left.\int_{E}D_{x_k}G(x,y)D_{i}((g^{ij}(y_0)-g^{ij}(y))D_j(u\eta))(y)dV(y)\right|_{x=y_0}\right|\nonumber\\
    \leq & \int_{B(y_0,rs)}C^2(d(y_0,y)^{1-n}+d(y_0,y)^{-n})d(y_0,y)\left(CN+\frac{C}{rs}\right)d\lambda(y)\\
    \leq &C_5(Nrs+1).\nonumber
    \end{align}
    \indent Since 
    $$0=\Delta u=\sum_{i,j}\frac{1}{\sqrt{\det g}}\frac{\partial }{\partial x_i}(g^{ij}\frac{\partial u}{\partial x_j}\sqrt{\det g})\text{ in }E,$$
    we have 
    $$\sum_{i,j}D_i(g^{ij}D_ju)=-\sum_{i,j}g^{ij}D_j\sqrt{\det g}D_iu=-\sum_ib^iD_iu\text{ in }E.$$
    We also have 
    $$\sum_{i,j}D_i(g^{ij}D_j(u\eta))=\sum_{i,j}(D_i(g^{ij}D_ju)\eta+D_i(g^{ij}D_j\eta)u)+2\sum_{i,j}g^{ij}D_j\eta D_iu,$$
    so for any $k$ and any $x\in E$, we get 
    \begin{align*}
    &\sum_{i,j}\int_{E}D_{x_k}G(x,y)D_{i}(g^{ij}D_j(u\eta))(y)dV(y)\\
    =&-\sum_i\int_E D_{x_k}G(x,y)b^i(y)D_iu(y)\eta(y)dV(y)\\
    &+\sum_{i,j}\int_E D_{x_k}G(x,y)D_i(g^{ij}D_j\eta)(y)u(y)dV(y)\\
    & +2\sum_{i,j}\int_ED_{x_k}G(x,y)g^{ij}(y)D_j\eta(y) D_iu(y)dV(y).
    \end{align*}
    Since $uD_j\eta =0$ on $\partial E$ and $D_j\eta=0$ in $B(y_0,\frac{rs}{2})$, then for any $k$ and any $x\in E$, we have  
    \begin{align*}
    &\sum_{i,j}\int_ED_{x_k}G(x,y)g^{ij}(y)D_j\eta(y) D_iu(y)dV(y)\\
    =& \sum_{i,j}\lim_{\epsilon\rw 0+}\int_{E\setminus B(y_0,\epsilon)}D_{x_k}G(x,y)g^{ij}(y)D_j\eta(y) D_iu(y)\sqrt{\det g(y)}d\lambda(y)\\
    =&-\sum_{i,j}\lim_{\epsilon\rw 0+}\int_{E\setminus B(y_0,\epsilon)}D_{y_i}(\sqrt{\det g(y)}D_{x_k}G(x,y)g^{ij}(y)D_j\eta(y))u(y)d\lambda(y)\\
    =&-\sum_{i,j}\int_{E}D_{y_i}\left(\sqrt{\det g(y)}D_{x_k}G(x,y)g^{ij}(y)D_j\eta(y)\right)u(y)d\lambda(y).
    \end{align*}
    Since $\supp(\eta)\subset B(y_0,rs)$, for any $k$, we then have 
    \begin{align*}
    &\sum_{i,j}\int_{E}D_{x_k}G(x,y)|_{x=y_0}D_{i}(g^{ij}D_j(u\eta))(y)dV(y)\\
    =& -\sum_i\int_E D_{x_k}G(x,y)|_{x=y_0}b^i(y)D_iu(y)\eta(y)dV(y)\\
    &+\sum_{i,j}\int_E D_{x_k}G(x,y)|_{x=y_0}D_i(g^{ij}D_j\eta)(y)u(y)dV(y)\\
    &-2\sum_{i,j}\int_{E}D_{y_i}\left(\sqrt{\det g(y)}D_{x_k}G(x,y)|_{x=y_0}g^{ij}(y)D_j\eta(y)\right)u(y)d\lambda(y).
    \end{align*}
    and then there is a constant $C_6>0$ such that  
    \begin{align}\label{gradient:equ3}
    &\left|\sum_{i,j}\int_{E}D_{x_k}G(x,y)|_{x=y_0}D_{i}(g^{ij}D_j(u\eta))(y)dV(y)\right|\\
    \leq&C^2N\int_{B(y_0,rs)}d(y_0,y)^{1-n}dV(y)+\int_{B(y_0,rs)}C^3\frac{d(y_0,y)^{1-n}}{rs}\left(1+\frac{1}{rs}\right)dV(y)\nonumber\\
    &+2\int_{B(y_0,rs)\setminus B(y_0,\frac{rs}{2})}C^4\frac{d(y_0,y)^{-n}}{rs}\left(2d(y_0,y)+1
    +d(y_0,y)\frac{1}{rs}\right)d\lambda(y)\nonumber\\
    \leq & C_6Nrs+C_6\left(1+\frac{1}{rs}\right)+C_6\int_{\frac{rs}{2}}^{rs}\frac{1}{rs}\left(2+\frac{1}{t}+\frac{1}{rs}\right)dt\nonumber\\
    =& C_6\left(Nrs+2+\frac{\frac{3}{2}+\ln2}{rs}\right)\nonumber.
    \end{align}
    Therefore, combining Formulas (\ref{gradient:equ1}), (\ref{gradient:equ2}) and (\ref{gradient:equ3}), there is a constant $C_7>0$ such that 
    $$\sum_k|D_{k}u(y_0)|\leq C_7\left(Nrs+1+\frac{1}{rs}\right).$$
    By the definition of gradient, we know there is a constant $C_8>0$ such that
    $$\frac{N}{2}\leq |\nabla u(y_0)|\leq C_8\left(Nrs+1+\frac{1}{rs}\right).$$ 
   Set $s:=\min\left\{\frac{1}{C_8r_0},1\right\}$, we know 
   $$N\leq \frac{C_8(1+4C_8)r_0}{r},$$ 
   and then the proof is complete. 
  \end{proof}
    \section{The proof of Theorem \ref{green function uniform}}
     In this section, we will give the proof of Theorem \ref{green function uniform}. Firstly, let us fix some notations. Let $(\overline M,g)$ be a compact Riemannian manifold with boundary (of dimension $n\geq 2$). It is clear that we may regard $M$ as a submanifold of a compact Riemannian manifold $M_0$ which has no boundary
     (for example, we may consider the double of $M$). The corresponding Riemannian metric on $M_0$ is also denoted by $g$.  Now we
     start to prove Theorem \ref{green function uniform}. For convenience, we restate it here.
     \begin{thm}[= Theorem \ref{green function uniform}]\label{thm:Green function}
     Let $G$ be the Green function of $\overline M$, then it has the following uniform estimate:
     \begin{itemize}
     \item[(i)] If $n\geq 3$, then there is a constant $C:=C(\overline M,M_0,g)>0$ such that for all $(x,y)\in (\overline M\times\overline M)\setminus\diag,$
     we have 
     $$ |G(x,y)|\leq Cd(x,y)^{2-n},\ |\nabla_y G(x,y)|\leq Cd(x,y)^{1-n}.$$
     \item[(ii)] If $n=2$, then there is a constant $C:=C(\overline M,M_0,g)>0$ such that for all $(x,y)\in (\overline M\times\overline M)\setminus\diag,$
     we have 
     $$ |G(x,y)|\leq C(1+|\ln d(x,y)|),$$
     $$|\nabla_y G(x,y)|\leq C(1+|\ln d(x,y)|)\frac{1}{d(x,y)}.$$
     \end{itemize}
     \end{thm}
    \begin{proof}
      We may assume $n\geq 3$. By Lemma \ref{green function:compact}, we may choose a function $G_0\in C^\infty((M_0\times M_0)\setminus\diag)$ such that 
     there is a constant $C_1:=C_1(M_0,g)>0$ satisfying (in the following, when we say a number is a constant,
     we mean it only depends on $\overline M,M_0,g$) 
     $$\Delta_y^{\distr}G_0(x,y)=\delta_x(y)-\frac{1}{|M_0|},\ \forall (x,y)\in M_0\times M_0;$$
     $$-C_1d(x,y)^{2-n}\leq G_0(x,y)\leq C_1,\ \forall\  (x,y)\in(M_0\times M_0)\setminus\diag.$$
     Choose  two open neighborhoods $U_1,U_2$ of $\overline M\times\overline M$ and choose 
      a function $\rho_1\in C^\infty(M_0\times M_0)$ such that 
      $$U_1\subset\subset U_2\subset\subset M_0\times M_0,\   \rho_1\geq 0,\ \rho_1|_{(M_0\times M_0)\setminus U_2}\equiv 1,\ \rho_1|_{U_1}\equiv 0.$$
     Set  
     $$\rho_2(x,y):=\frac{1}{|M_0|}-\frac{\rho_1(x,y)}{\int_{M_0}\rho_1(x,z)dV(z)},\ \forall (x,y)\in M_0\times M_0,$$
     then we know 
     $$\rho_2(x,\cdot)\in C^\infty(M_0),\ \int_{M_0}\rho_2(x,\cdot)dV=0,\ \forall x\in M_0.$$
     By Lemma \ref{Aubin:Theorem 4.7}, for any $x\in M_0$, we may find a function $\rho_3(x,\cdot)\in C^\infty(M_0)$ such that 
     $$\Delta_y \rho_3(x,y)=\rho_2(x,y),\ \int_{M_0}\rho_3(x,y)dV(y)=0.$$
     By Lemma \ref{Aubin:Theorem 4.7b}, we know there is a constant $C_2>0$ such that 
     $$\sup_{y\in M_0}|\rho_3(x,y)|\leq C_2\sup_{y\in M_0}|\rho_2(x,y)|,\ \forall x\in M_0.$$
     Let 
     $$C_3:=C_2\sup_{x,y\in M_0}|\rho_2(x,y)|,$$
     then we know 
     $$\sup_{x,y\in M_0}|\rho_3(x,y)|\leq C_3.$$
     For any $x\in M_0$, set 
     $$\Gamma(x,y):=G_0(x,y)-\rho_3(x,y),\ \forall y\in M_0\setminus\{x\},$$
     then we get 
     $$\Delta_y^{\distr}\Gamma(x,y)=\delta_x(y),\ \forall (x,y)\in U_1,$$
     $$-C_1d(x,y)^{2-n}-C_3\leq \Gamma(x,y)\leq C_1+C_3,\ \forall (x,y)\in (M_0\times M_0)\setminus\diag.$$
     For any $x\in M$, using Lemma \ref{Aubin:Theorem 4.8}, we can find a function $\Phi(x,\cdot)\in C^\infty(\overline M)$ such that 
     $$\left\{\begin{array}{ll}
     \Delta_y\Phi(x,y)=0 &\text{ in } M,\\
     \Phi(x,y)=\Gamma(x,y) & \text{ on }\partial M.
     \end{array}\right.$$
     By the uniqueness of the Green function, we know 
     $$G(x,y)=\Gamma(x,y)-\Phi(x,y),\ \forall (x,y)\in M\times \overline M.$$
     Applying the maximum principle, we have 
     $$\sup_{x\in M,y\in\overline M}\Phi(x,y)\leq C_1+C_3,$$
     then we get 
     $$G(x,y)\geq -C_1 d(x,y)^{2-n}-C_1-2C_3,\ \forall (x,y)\in  (M\times\overline{M})\setminus \diag.$$
     Since $M$ is bounded and by continuity, there is a constant $C_4>0$ satisfying 
     $$G(x,y)\geq - C_4 d(x,y)^{2-n},\ \forall (x,y)\in  (\overline M\times\overline{M})\setminus \diag,$$
     which completes the proof of the first part of (i).\\
     \indent Now we turn to the proof of the uniform estimate of the gradient of the Green function of Part (i).
     Indeed, using Lemma \ref{gradient estimate 1} and Lemma \ref{gradient estimate 2}, we can easily  get what we wanted by imitating the proof of Theorem
     \ref{estimate:euclidean}.
     The details is stated as follows. \\
     \indent We may assume $x,y\in M$. For any $x\in M$, set $\delta(x):=\inf_{y\in\partial M}d(x,y).$ It is easy to see that 
     there is $r_0>0$ such that for any $z\in\partial M$ and any $0<r<r_0$, there is an open geodesic ball $B^z(r)$ with radius $r$ satisfying 
      $$B^z(r)\subset M_0\setminus\overline{M},\ \overline{B^z(r)}\cap \partial M=\{z\}.$$ 
      Shrinking $r_0$ if necessary, we may assume $10r_0$ is less than the injective radius of $M_0.$\\
     {\bf Claim $\clubsuit$:} There is a constant $C_6>0$ such that 
     $$|G(x,y)|\leq C_6d(x,y)^{1-n}\delta(x),\ \forall\ (x,y) \in (M\times M)\setminus\diag.$$
    \indent To prove the Claim, we fix $(x,y)\in (M\times M)\setminus\diag$ and consider two cases.\\
    {\bf Case 1:} $\delta(x)<\min\left\{\frac{d(x,y)}{8},r_0\right\}$.\\
    \indent Set $r:=\min\left\{\frac{d(x,y)}{8},r_0\right\}$. Choose $z_x\in\partial M,\ x^*\in M_0\setminus\overline{M}$ such that
    $$B^{z_x}(r)=B(x^*,r),$$
    and set $E:=B(x^*,2r)\setminus \overline{B(x^*,r)}.$ Using Lemma \ref{Aubin:Theorem 4.8}, we may find a function $u\in C^\infty(\overline E)$ satisfying 
    $$
    \left\{\begin{array}{ll}
      \Delta u=0&\text{ in }E,\\
       u=0 & \text{ on }\partial B(x^*,r),\\
       u=1 & \text{ on }\partial B(x^*,2r).
    \end{array}\right.
    $$
    By Lemma \ref{gradient estimate 2}, there is a constant $C_5>0$ such that 
    $$\sup_{E}|\nabla u|\leq \frac{C_5}{r}.$$
    As $z_x\in\partial B(x^*,r)\cap\partial M$, then by the mean value theorem, we know  
    $$u(x)=|u(x)-u(z_x)|\leq C_5\frac{d(x,z_x)}{r}=\frac{C_5\delta(x)}{r}.$$
    For any $z\in\partial B(x^*,2r)\cap M$, we have 
    $$d(z_x,z)\leq d(z_x,x^*)+d(x^*,z)=3r,\ d(x,z)\leq d(x,z_x)+d(z_x,z)\leq 4r,$$
    $$d(z,y)\geq d(x,y)-d(x,z)\geq d(x,y)-4r\geq \frac{d(x,y)}{2},$$
    so we know 
    $$  |G(z,y)|\leq C_4d(z,y)^{2-n}\leq C_42^{n}d(x,y)^{2-n}u(z). $$
    Note that for any $z\in \partial M$, we have $G(z,y)=0$.  Applying the maximum principle to the set $M\cap E\ni x$, we get 
    \begin{align}\label{gradient:equ4}
    |G(x,y)|&\leq C_42^{n}d(x,y)^{2-n}u(x)\leq C_4C_52^{n}\frac{d(x,y)^{2-n}}{r}\delta(x)\\
          &\leq C_4C_52^{n}\left(8+\frac{\diam(M)}{r_0}\right)d(x,y)^{1-n}\delta(x),\nonumber
    \end{align}\
    where $\diam(M)$ is the diameter of $M$.\\
    {\bf Case 2:} $\delta(x)\geq \min\left\{r_0,\frac{d(x,y)}{8}\right\}$.\\
    \indent In this case, we have 
    $$\frac{d(x,y)}{\delta(x)}\leq 8+\frac{\diam(M)}{r_0},$$
    so we get 
    \begin{equation}\label{gradient:equ5}
    |G(x,y)|\leq C_4d(x,y)^{2-n}\leq C_4\left(8+\frac{\diam(M)}{r_0}\right)d(x,y)^{1-n}\delta(x).
    \end{equation}
    \indent Combining {\bf Case 1} and {\bf Case 2}, especially Inequalities (\ref{gradient:equ4}) and (\ref{gradient:equ5}),
     we know {\bf Claim $\clubsuit$} holds with a constant 
     $$C_6:=\max\left\{C_4C_52^{n}\left(8+\frac{\diam(M)}{r_0}\right),C_4\left(8+\frac{\diam(M)}{r_0}\right)\right\}.$$
     \indent Now we start to prove the last part of (iv). Fix $(x,y)\in (M\times M)\setminus\diag$, and we also consider two cases.\\
     {\bf Case 1:} $\delta(y)\leq d(x,y)$.\\
    \indent In this case, we know $G(x,\cdot)$ is a harmonic function in $B(y,\delta(y))$. By Lemma \ref{gradient estimate 1}, 
    there exists a constant $C_7>0$ such that 
    $$|\nabla_y G(x,y)|\leq \frac{C_7}{\frac{1}{2}\delta(y)}\sup_{B\left(y,\frac{1}{2}\delta(y)\right)}|G(x,\cdot)|.$$
    For any $z\in B\left(y,\frac{1}{2}\delta(y)\right)$, we have 
    $$d(x,z)\geq d(x,y)-d(y,z)\geq d(x,y)-\frac{1}{2}\delta(y)\geq \frac{d(x,y)}{2},$$
    $$\delta(z)\leq \delta(y)+d(y,z)\leq 2\delta(y),$$
    then by {\bf Claim $\clubsuit$}, we get 
    $$|G(x,z)|\leq C_6d(x,z)^{1-n}\delta(z)\leq C_62^{n}d(x,y)^{1-n}\delta(y),$$ 
    and thus we know  
    \begin{equation}\label{gradient:equ6}
    |\nabla_y G(x,y)|\leq 2^{n+1}C_6C_7d(x,y)^{1-n}.
    \end{equation}
    {\bf Case 2:} $\delta(y)>d(x,y)$. \\
    \indent In this case, we know $G(x,\cdot)$ is a harmonic function in $B(y,d(x,y))$. By Lemma \ref{gradient estimate 1}, 
    there is a constant $C_8>0$ such that  
    $$|\nabla_y G(x,y)|\leq \frac{C_8}{\frac{1}{2}d(x,y)}\sup_{B\left(y,\frac{1}{2}d(x,y)\right)}|G(x,\cdot)|.$$
    For any $z\in B\left(y,\frac{1}{2}d(x,y)\right)$, we have 
    $$d(x,z)\geq d(x,y)-d(y,z)\geq \frac{1}{2}d(x,y),$$
    so we have 
    $$|G(x,z)|\leq C_4d(x,z)^{2-n}\leq C_42^{n-2}d(x,y)^{2-n},$$
    and then we get 
    \begin{equation}
    |\nabla_y G(x,y)|\leq 2^{n}C_4C_8d(x,y)^{1-n}.
    \end{equation}
    \indent By Inequalities (\ref{gradient:equ5}) and (\ref{gradient:equ6}), we know the last part of (iv) holds with 
    $$C:=\max\left\{2^{n+1}C_6C_7,2^{n}C_4C_8\right\}.$$
    \indent The proof is complete.
    \end{proof}
    \begin{remark}\label{remark:continuous}
    If we know $\Phi\in C^1(\overline M\times \overline M)$ before, then the above proof can be reduced a lot.
    However, we didn't find an easy way to see it directly. But, it can be easily derived from the uniform estimate of Green functions  by using the Green representation formula. See \cite[Lemma 4.2]{Gilbarg} and its proof.
    \end{remark}
    Now we also give a sketch of a second proof for the uniform estimates of Green functions when $n\geq 3$.
    \begin{proof}
     We may assume both $M$ and $M_0$ are connected. Let $H(t,x,y)$ be the heat kernel of $M$. 
     Since $M\subset\subset M_0$ and $M_0\setminus\overline{M}\neq \emptyset$, then $M$ satisfies the Faber-Krahn inequality by \cite[Corollary 15.12]{Grigoryan}.
     For any $x,y\in M,\ x\neq y$, using \cite[Corollary 15.17, Formula 15.49]{Grigoryan}, there is a constant $C:=C(M,g)>0$ such that 
      \begin{align*}
      -G(x,y)&=\int_{0}^\infty H(t,x,y)dt\leq C\int_0^\infty t^{-\frac{n}{2}}e^{-\frac{d(x,y)^2}{8t}}dt\\
      &= Cd(x,y)^{2-n}\int_0^\infty\frac{(8s)^{\frac{n}{2}}}{8s^2}e^{-s}ds\\
      &=8^{\frac{n}{2}-1}C\Gamma\left(\frac{n}{2}-1\right)d(x,y)^{2-n},
     \end{align*}
     where $\Gamma(\cdot)$ denotes the Gamma function.
    \end{proof}
    \section{The proofs of Theorem \ref{thm:graident boundary}, \ref{thm:laplace boundary}, \ref{thm:laplace compact1}
    and \ref{thm:bochner martinella}}
    In this section, we prove Theorem \ref{thm:graident boundary}, \ref{thm:laplace boundary}, \ref{thm:laplace compact1}
    and \ref{thm:bochner martinella}. To do this, we need the following $L^p$-estimates of the Green functions and its gradients.
    \begin{lem}\label{lem:gradient 1}
    Let $(\overline M,g)$ be a compact Riemannian manifold with boundary (of dimension $n\geq 2$), and let $G$ be the Green function.
    For any $a>0$ such that the exponents of the following integrands are nonnegative, we have 
    \begin{itemize}
      \item[(i)]  If $n\geq 3$, then there is a constant $C:=C(\overline M,g)>0$
    such that 
    $$\sup_{x\in M}\int_{M}|G(x,y)|^{\frac{-n+a}{2-n}}dV(y)\leq \frac{C}{a},\ \sup_{x\in M}\int_{M}|\nabla_yG(x,y)|^{\frac{-n+a}{1-n}}dV(y)\leq \frac{C}{a},$$
    $$\sup_{x\in M}\int_{\partial M}|\nabla_yG(x,y)|^{\frac{-n+1+a}{1-n}}dS(y)\leq \frac{C}{a}.$$
      \item[(ii)] If $n=2$, then there is a constant $C:=C(\overline M,g)>0$ such that 
    $$\sup_{x\in M}\int_{M}|\nabla_yG(x,y)|^{2-a}dV(y)\leq \frac{C}{a},\ \sup_{x\in M}\int_{\partial M}|\nabla_yG(x,y)|^{1-a}dS(y)\leq \frac{C}{a}.$$
    \item[(iii)] If $n=2$, then there is a constant $C:=C(\overline M,g)>0$ such that for any $1\leq p<\infty$, we have 
    $$\sup_{x\in M}\int_{M}|G(x,y)|^{p}dV(y)\leq p^pC^{p+1}.$$
    \end{itemize}
    \end{lem}
    \begin{proof}
    (i) By Theorem \ref{thm:Green function}, we know there is a constant $C_1:=C_1(\overline M,g)>0$ such that 
    for all  $(x,y)\in(\overline M\times \overline M)\setminus\diag,$ we have 
    $$|G(x,y)|\leq C_1d(x,y)^{2-n},\ |\nabla_y G(x,y)|\leq C_1d(x,y)^{1-n}.$$
    By Lemma 3.1 and Lemma 3.4 of \cite{DJQ}, we know the conclusion holds by choosing a normal coordinate as $M$ is compact.\\
    (ii) By Theorem \ref{thm:Green function}, we know there is a constant $C_2:=C_2(\overline M,g)>0$ such that 
    for all  $(x,y)\in(\overline M\times \overline M)\setminus\diag,$ we have 
    $$|G(x,y)|\leq C_2(1+|\ln d(x,y)|),\ |\nabla_y G(x,y)|\leq C_2(1+|\ln d(x,y)|)\frac{1}{d(x,y)}.$$
    Clearly, there is a constant $C_3:=C_3(\diam(M))>0$ such that
    $$1+|\ln r|\leq \frac{C_3}{br^b},\ \forall\ 0<r\leq \diam(M),\ \forall\ 0<b\leq 1.$$
    Set $b:=\frac{a}{2(2-a)}$, then for all  $(x,y)\in(\overline M\times \overline M)\setminus\diag,$ 
    we have  
    $$ |\nabla_y G(x,y)|\leq \frac{C_2C_3}{d(x,y)^{1+b}b}\leq \frac{4C_2C_3}{a}d(x,y)^{-1+\frac{a}{2(a-2)}}.$$
    By Lemma 3.1 and Lemma 3.4 of \cite{DJQ}, we know the conclusion holds.\\
    (iii)  For any $0<b<1$, we know 
    $$|G(x,y)|^p\leq C_2^p(1+|\ln d(x,y)|)^p\leq \left(\frac{C_2C_3}{b}\right)^pd(x,y)^{-pb}$$
    for all  $(x,y)\in(\overline M\times \overline M)\setminus\diag.$
    Set $b:=\frac{1}{p}$, then by \cite[Lemma 3.1]{DJQ}, we know there is a constant $C_4:=C_4(\overline M,g)>0$ such that 
    $$\sup_{x\in M}\int_{M}d(x,y)^{-pb}dV(y)\leq C_4.$$
    Thus, we know 
    $$\sup_{x\in M}\int_{M}|G(x,y)|^pdV(y)\leq (pC_2C_3)^pC_4.$$
    \end{proof}
     Similar results also hold when $M$ has no boundary by Lemma \ref{green function:compact}.
     \begin{lem}\label{lem:laplace 1}
    Let $(M,g)$ be a compact Riemannian manifold without boundary, and let $G$ be the Green function.
    \begin{itemize}
      \item[(i)]  If $n\geq 3$, then for any $0<a\leq n$, there is a constant $C:=C( M,g)>0$
    such that 
    $$\sup_{x\in M}\int_{M}|G(x,y)|^{\frac{-n+a}{2-n}}dV(y)\leq \frac{C}{a}.$$
    \item[(ii)] If $n=2$, then there is a constant $C:=C( M,g)>0$ such that for any $1\leq p<\infty$, we have 
    $$\sup_{x\in M}\int_{M}|G(x,y)|^{p}dV(y)\leq p^pC^{p+1}.$$
    \end{itemize}
    \end{lem} 
    Another useful simple lemma is the following
    \begin{lem}\label{lem:integration by parts 1}
    Let $(\overline M,g)$ be a compact Riemannian manifold with  a smooth boundary defining function $\rho$, and let $k\in\mathbb{N}$. For any $(k-1)$-form $s_1$ and $k$-form $s_2$ which are of class $C^1$ on $\overline M$, we have 
    $$\int_{M}\langle ds_1,s_2\rangle dV-\int_{M}\langle s_1,d^* s_2\rangle dV=\int_{\partial M}\langle d\rho\wedge s_1,s_2\rangle \frac{dS}{|\nabla\rho|}.$$
    \end{lem}
    \begin{proof}
    We may assume $k\geq 1$.
    By Formula (3.5) and Theorem 3.9 of \cite[Chapter VI]{Dem}, we have
    $$d^*s_2=(-1)^{nk+n+1}*d*s_2,\ **d*s_2=(-1)^{(n+1)(n-k+1)}d*s_2,$$
    and then 
    \begin{align*}
    &\langle ds_1,s_2\rangle dV-\langle s_1,d^* s_2\rangle dV\\
    =&ds_1\wedge *s_2-s_1\wedge *d^*s_2\\
    =&  ds_1\wedge *s_2+(-1)^{nk+n}s_1\wedge **d*s_2\\
    =&ds_1\wedge *s_2+(-1)^{nk+n}(-1)^{(n+1)(n-k+1)}s_1\wedge d*s_2\\
    =&  ds_1\wedge *s_2+(-1)^{k-1}s_1\wedge d*s_2\\
    =& d(s_1\wedge *s_2).
    \end{align*}
    By the Stokes theorem, we have 
    $$\int_{M}\langle ds_1,s_2\rangle dV-\langle s_1,d^* s_2\rangle dV=\int_{\partial M}s_1\wedge *s_2.$$
    Now it suffices to prove 
    $$(s_1\wedge *s_2)|_{\partial M}=\langle d\rho\wedge s_1,s_2\rangle \frac{dS}{|\nabla\rho|}.$$
    This problem is local, we may taking a local orthonormal frame $e_1,\cdots,e_n$ of $TM$ such that $e_n$ is the outward unit normal vector of $\partial M$.
    Let $e^1,\cdots,e^n$ be the dual frame. By linearity, we may assume $s_2=e^{j_1}\wedge\cdots\wedge e^{j_k}$ for some $1\leq j_1<\cdots<j_k\leq n$.
    Write 
    $$s_1=\sum_{1\leq i_1<\cdots<i_{k-1}\leq n}a_{i_1\cdots i_{k-1}}e^{i_1}\wedge\cdots\wedge e^{i_{k-1}},$$
    then 
    $$*s_2=\sgn(J,J^c)e^{J^c},\ s_1\wedge *s_2=\sum_{p=1}^k \sgn(J,J^c)a_{j_1\cdots\widehat{j_p}\cdots j_k}e^{J\setminus\{j_p\}}\wedge e^{J^c},$$
    where $\sgn$ denotes the sign of a permutation, $J^c:=\{1,\cdots,n\}\setminus J$. Since the restriction of $e^n$ on $\partial M$ is zero, we know 
    $$(s_1\wedge *s_2)|_{\partial M}=(-1)^{n-k}a_{j_1\cdots j_{k-1}}e^1\wedge\cdots\wedge e^{n-1}.$$
    Note that $\nabla\rho$ is parallel to $e_n$, then we have 
    $$d\rho= |\nabla\rho|e^n,$$
    so we get 
    $$\langle d\rho \wedge s_1,s_2\rangle=(-1)^{k-1}|\nabla\rho|a_{j_1\cdots j_{k-1}}.$$
    Note also that $dS=(-1)^{n-1}e^1\wedge\cdots\wedge e^{n-1}$, thus we have 
    $$(s_1\wedge *s_2)|_{\partial M}=\langle d\rho \wedge s_1,s_2\rangle \frac{dS}{|\nabla\rho|},$$
    which completes the proof.
    \end{proof}
    Using Green functions, we may give an integral representation, i.e.
    \begin{thm}\label{thm:gradient 1}
    Let $(\overline M,g)$ be a compact Riemannian manifold with boundary. For any $f\in C^1(\overline M),\ x\in M$, we have 
     $$f(x)=-\int_{M}\langle d_yG(x,y),df(y)\rangle dV(y)+\int_{\partial M}\frac{\partial G(x,y)}{\partial\nve_y}f(y)dS(y),$$
    \end{thm}
    \begin{proof}
    By density, we may assume $f\in C^2(\overline M)$. By definition, we know $-\Delta=\Delta_d=dd^*+d^*d$. 
    Fix $x\in M$. By Lemma \ref{lem:integration by parts 1}
    and note that $G(x,\cdot)|_{\partial M}=0$, we have 
    \begin{align*}
     & \int_{M} G(x,y)\Delta f(y)dV(y)+\int_{M} \langle d_yG(x,y),df(y)\rangle dV(y)\\
     =-&\lim_{\epsilon\rw 0+}\int_{M\setminus B(x,\epsilon)} \langle G(x,y),d^*df(y)\rangle dV(y)+\int_{M} \langle d_yG(x,y),df(y)\rangle dV(y)\\
     =&\lim_{\epsilon\rw 0+}(-1)^{n+1}\int_{\partial B(x,\epsilon)}\langle G(x,y)d\rho(y),df(y)\rangle\frac{dS(y)}{|\nabla\rho(y)}.
    \end{align*}
    By Theorem \ref{thm:Green function}, we may find a constant $C:=C(\overline M,g)>0$ such that for all $y\in \overline M\setminus\{x\}$,
    we have
    $$|G(x,y)|\leq C\cdot \left\{\begin{array}{ll}
    d(x,y)^{2-n} & \text{ if }n\geq 3,\\
    1+|\ln d(x,y)| & \text{ if }n=2,
    \end{array}\right.$$
    and thus it is clear that 
    $$\lim_{\epsilon\rw 0+}\int_{\partial B(x,\epsilon)}\langle G(x,y)d\rho(y),df(y)\rangle\frac{dS(y)}{|\nabla\rho(y)}=0.$$
    By Lemma \ref{green function:compact}, we get 
    \begin{align*}
     f(x)&=\int_{M}G(x,y)\Delta f(y)dV(y)+\int_{\partial M}\frac{\partial G(x,y)}{\partial\nve_y}f(y)dS(y),\\
     &=-\int_{M} \langle d_yG(x,y),df(y)\rangle dV(y)+\int_{\partial M}\frac{\partial G(x,y)}{\partial\nve_y}f(y)dS(y).
    \end{align*}
    \end{proof}
    A direct consequence of Theorem \ref{thm:gradient 1} is  the following
    \begin{cor}\label{cor:gradient 1}
    Let $(\overline M,g)$ be a compact Riemannian manifold with boundary. For any $f\in C^1(\overline M),\ x\in M$, we have 
     $$|f(x)|\leq \int_{M}|\nabla_y G(x,y)|\cdot |\nabla f(y)|dV(y)+\int_{\partial M}|\nabla_yG(x,y)|\cdot|f(y)|dS(y),$$
    \end{cor}
    \begin{proof}
    Note that for any $h\in C^1(\overline M)$, we have $|dh|=|\nabla h|,$
    then the result follows from Cauchy-Schwarz Inequality. 
    \end{proof}
    Same idea as in the proof of Theorem \ref{thm:gradient 1}, we can prove Theorem \ref{thm:bochner martinella}.
    \begin{thm}[= Theorem \ref{thm:bochner martinella}]
    Let $(M,h)$ be a balanced manifold, $\Omega\subset\subset M$ with  smooth boundary, and let $G$ be the Green 
    function for the Laplace operator on $\overline\Omega$. Then we have 
    $$f(z)=-2\int_{\Omega}\langle \bar\partial f(w),\bar\partial_w G(z,w)\rangle dV(w)+\int_{\partial \Omega}\frac{\partial G(z,w)}{\partial\nve_w}f(w)dS(w)$$
    for all $f\in C^0(\overline\Omega)\cap C^1(\Omega)$ such that $\bar\partial f$ has a continuous continuation to the boundary.
    \end{thm}
    \begin{proof}
    By the usual approximation technique, we may assume $f\in C^2(\overline\Omega)$.
    Note that $-\Delta=2\Delta_{\bar\partial}=2\bar\partial\bar\partial^*+2\bar\partial^*\bar\partial$ as $(M,h)$ is balanced, then similar to the proof of Theorem \ref{thm:gradient 1}, for any $z\in M$, 
    by Theorem \ref{thm:Green function}, we have 
    \begin{align*}
     f(z)&=\int_{\Omega}G(z,w)\Delta f(w)dV(w)+\int_{\partial \Omega}\frac{\partial G(z,w)}{\partial\nve_w}f(y)dS(y),\\
     &=-2\int_{\Omega}\langle\bar\partial^*\bar\partial f(w),G(z,w)\rangle dV(w)+\int_{\partial \Omega}\frac{\partial G(z,w)}{\partial\nve_w}f(w)dS(w),\\
     &=-2\int_{\Omega} \langle \bar\partial f(w),\bar\partial_w G(z,w)\rangle dV(w)+\int_{\partial \Omega}\frac{\partial G(x,y)}{\partial\nve_y}f(w)dS(w).
    \end{align*}
    \end{proof}
     Theorem \ref{thm:graident boundary} is a consequence of the following lemma by Corollary \ref{cor:gradient 1}.
   \begin{lem}\label{lem:basic lemma 2}
   Let $p,q,r$ satisfy
        $$
         1\leq p<\infty,\ 1\leq q<\infty,\ 1\leq r<\infty,\ q(n-p)<np,\ q(n-1)<nr,
        $$
     then
     \begin{itemize}
           \item[(i)] There exists a constant $\delta:=\delta(\overline M,g,p,q)>0$ such that
          $$
           \delta\|B_M f\|_{L^q(M)}\leq \|f\|_{L^p(M)},\ \forall f\in C^0(\overline M),
           $$
          where
          $$
          B_M f(x):=\int_{M}|\nabla_y G(x,y)|\cdot |f(y)|dV(y),\ \forall x\in M.
          $$
           \item[(ii)] There exists a constant $\delta:=\delta(\overline M,g,q,r)>0$ such that
            $$
            \delta\|B_{\partial M} f\|_{L^q(M)}\leq \|f\|_{L^r(\partial M)},\ \forall f\in C^0(\overline M),
            $$
          where
          $$
          B_{\partial M}f(x):=\int_{\partial M}|\nabla_yG(x,y)|\cdot|f(y)|dS(y),\ \forall x\in M.$$
    \end{itemize}
    \end{lem}
    \begin{proof}
        (i) For any $0<a\leq n,$ set
         $$
         N(a,M):=\sup_{x\in M}\int_{M}|\nabla_y G(x,y)|^{\frac{-n+a}{1-n}}dV(y).$$
         By Lemma \ref{lem:gradient 1}, there is a constant $C_1:=C(\overline M,g)>0$ such that 
         $$N(a,M)\leq \frac{C_1}{a}.$$
         Since $M$ is bounded, then by H\"older's Inequality, we may assume $q\geq p$. We moreover assume $p>1$
        since the case $p=1$ can be treated by taking a limit if we note that 
        $$\|f\|_{L^1(M)}=\lim_{p\rw 1+}\|f\|_{L^p(M)}$$
        and the corresponding constant (in the last line of this part) is uniformly bounded as $p\rw 1+$.  Set
        $$
        b:=\frac{1}{2}\left(\max\left\{0,\frac{n-p}{p(n-1)}\right\}+\min\left\{\frac{n}{q(n-1)},1\right\}\right),
        $$
        then by assumption, we have
        $$
        0<\frac{p(1-n)(1-b)}{p-1}+n\leq n,\ 0<qb(1-n)+n\leq n.
        $$
        Choose $a,q_0,r_0$ such that
        $$
        qa=p,\ q_0(1-a)=p,\ \frac{1}{q}+\frac{1}{q_0}+\frac{1}{r_0}=1,
        $$
        where if $p=q$, then $q_0=\infty$.
         Fix any $f\in C^0(\overline{M})$, then for any $ x\in M,$ we have (use H\"older's Inequality for three functions)
        \begin{align*}
        & |B_M f(x)| \\
        =&\int_{\Omega}|f(y)|^{a}|\nabla_y G(x,y)|^{b}|f(y)|^{1-a}|\nabla_y G(x,y)|^{1-b}dV(y)  \\
        \leq &\left[\int_{M}|f(y)|^{p}|\nabla_y G(x,y)|^{qb}dV(y)\right]^{\frac{1}{q}} \left[\int_{M}|f(z)|^pdV(z)\right]^{\frac{q-p}{qp}}  \\
        &\cdot \left[\int_{M}|\nabla_wG(x,w)|^{\frac{p(1-b)}{p-1}}dV(w)\right]^{\frac{p-1}{p}},
        \end{align*}
         which implies
        \begin{align*}
        &|B_M f(x)|^q \\
        \leq & \|f\|_{L^p(M)}^{q-p}\int_{M}|f(y)|^{p}|\nabla_y G(x,y)|^{qb}dV(y)\left[\int_{M}|\nabla_wG(x,w)|^{\frac{p(1-b)}{p-1}}dV(w)\right]^{\frac{q(p-1)}{p}}  \\
        \leq &N\left(\frac{p(1-n)(1-b)}{p-1}+n,M\right)^{\frac{q(p-1)}{p}}\|f\|_{L^p(M)}^{q-p}\int_{M}|f(y)|^{p}|\nabla_y G(x,y)|^{qb}dV(y),
        \end{align*}
        so (by Fubini's Theorem)
        \begin{align*}
        & \|B_M f\|_{L^q(M)}^q \\
        \leq &N(qb(1-n)+n,M)N\left(\frac{p(1-n)(1-b)}{p-1}+n,M\right)^{\frac{q(p-1)}{p}}\|f\|_{L^p(M)}^{q} .
        \end{align*}
        A direct computation shows that
        \begin{align*}
        & N(qb(1-n)+n,M)N\left(\frac{p(1-n)(1-b)}{p-1}+n,M\right)^{\frac{q(p-1)}{p}}\\
        \leq & \frac{C_1}{qb(1-n)+n}\left(\frac{C_1(p-1)}{p-n-pb(1-n)}\right)^{\frac{q(p-1)}{p}}.
        \end{align*}
        (ii) For any $0<a\leq n-1,$ set
         $$
         N(a,\partial M):=\sup_{x\in M}\int_{\partial M}|\nabla_y G(x,y)|^{\frac{-n+1+a}{1-n}}dS(y),
         $$
         then by Lemma \ref{lem:gradient 1}, there is a constant $C_2:=C_2(\overline M,g)>0$ such that 
         $$N(a,\partial M)\leq \frac{C_2}{a}.$$
         Similar to (i), we may assume $q\geq r>1$. Set
        $$b:=\frac{1}{2}\left(\min\left\{\frac{n}{q(n-1)},1\right\}+\frac{1}{r}\right),$$
        then we have
        $$
          0<qb(1-n)+n\leq n,\ 0<\frac{(rb-1)(n-1)}{r-1}\leq n-1.$$
        Choose $a,q_0,r_0$ such that
        $$
        qa=r,\ q_0(1-a)=r,\ \frac{1}{q}+\frac{1}{q_0}+\frac{1}{r_0}=1,
        $$
        where if $q=r$, then $q_0=\infty$.
         Fix  $f\in C^0(\overline{M})$, then for any $ x\in M,$ we have
        \begin{align*}
         & |B_{\partial M}f(x)| \\
         = &\int_{\partial M}|f(y)|^a|\nabla_y G(x,y)|^b |f(y)|^{1-a}|\nabla_y G(x,y)|^{1-b}dS(y)  \\
         \leq& \left[\int_{\partial M}|f(y)|^{qa}|\nabla_y G(x,y)|^{qb}dS(y)\right]^{\frac{1}{q}}
          \left[\int_{\partial M}|f(z)|^{q_0(1-a)}dS(z)\right]^{\frac{1}{q_0}}  \\
         &\cdot\left[\int_{\partial M}|\nabla_wG(x,w)|^{\frac{r(1-b)}{r-1}}dS(w)\right]^{\frac{r-1}{r}}  \\
         \leq & N\left(\frac{(rb-1)(n-1)}{r-1},\partial M\right)^{\frac{r-1}{r}} \left[\int_{\partial M}|f(y)|^{r}|\nabla_y G(x,y)|^{qb}dS(y)\right]^{\frac{1}{q}}  \\
         &\cdot \left[\int_{\partial M}|f(z)|^{r}dS(z)\right]^{\frac{q-r}{qr}}  ,
        \end{align*}
        so we get
        \begin{align*}
        &|B_{\partial M}f(x)|^{q} \\
        \leq & N\left(\frac{(rb-1)(n-1)}{r-1},\partial M\right)^{\frac{q(r-1)}{r}}
        \left[\int_{\partial M}|f(y)|^{r}|\nabla_y G(x,y)|^{qb}dS(y)\right]^{\frac{1}{q}}  \\
         &\cdot \left[\int_{\partial M}|f(z)|^{r}dS(z)\right]^{\frac{q-r}{qr}},
        \end{align*}
         then
        \begin{align*}
        & \int_{M}|B_{\partial M} f(x)|^{q}dV(x) \\
        \leq & N\left(\frac{(rb-1)(n-1)}{r-1},\partial M\right)^{\frac{q(r-1)}{r}} \int_{(x,y)\in M\times\partial M}
        |\nabla_y G(x,y)|^{qb}dV(x)dS(y)  \\
        &\cdot \|f\|_{L^r(\partial M)}^q \\
        \leq & N(qb(1-n)+n,M)|\partial M| N\left(\frac{(rb-1)(n-1)}{r-1},\partial M\right)^{\frac{q(r-1)}{r}}\|f\|_{L^r(\partial M)}^q.
        \end{align*}
        A direct computation shows that
        \begin{align*}
        &N(qb(1-n)+n,M) N\left(\frac{(rb-1)(n-1)}{r-1},\partial M\right)^{\frac{q(r-1)}{r}}\\
        \leq &\frac{C_1}{qb(1-n)+n}\left(\frac{(r-1)C_2}{(rb-1)(n-1)}\right)^{\frac{q(r-1)}{r}}.
        \end{align*}
        \indent The proof is completed.
    \end{proof}
    To prove Theorem \ref{thm:laplace boundary},  using the integral representation as in Theorem \ref{thm:Green function}, it suffices to prove the following lemma
     \begin{lem}\label{lem:basic lemma 3}
       Let $p,q,r$ satisfy
            $$
             1\leq p<\infty,\ 1\leq q<\infty,\ 1\leq r<\infty,\ q(n-2p)<np,\ q(n-1)<nr,
            $$
       and if $n=2$, we moreover require that $p>1$ if $q>1$. Then
    \begin{itemize}
           \item[(i)] There exists a constant $\delta:=\delta(M,g,p,q)>0$ such that
          $$
           \delta\|B_M f\|_{L^q(M)}\leq \|f\|_{L^p(M)},\ \forall f\in C^0(\overline M),
           $$
          where
          $$
          B_M f(x):=\int_{M}|G(x,y)|\cdot |f(y)|dV(y),\ \forall x\in M.
          $$
           \item[(ii)] There exists a constant $\delta:=\delta(M,g,q,r)>0$ such that
            $$
            \delta\|B_{\partial M} f\|_{L^q(M)}\leq \|f\|_{L^r(\partial M)},\ \forall f\in C^0(\overline M),
            $$
          where
          $$
          B_{\partial M}f(x):=\int_{\partial M}|\nabla_yG(x,y)|\cdot|f(y)|dS(y),\ \forall x\in M.$$
    \end{itemize}
    \end{lem}
    \begin{proof}
    Part (ii) has been proved in Lemma \ref{lem:basic lemma 2}, so we only need to prove Part (i).
    This is similar to the proof of Lemma \ref{lem:basic lemma 2}. We consider two cases.\\
    {\bf Case 1:} $n\geq 3$.\\
    \indent For any $0<a\leq n,$ set
         $$
         N(a,M):=\sup_{x\in M}\int_{M}|G(x,y)|^{\frac{-n+a}{2-n}}dV(y).$$
         By Lemma \ref{lem:gradient 1}, there is a constant $C_1:=C_1(\overline M,g)>0$ such that 
         $$N(a,M)\leq \frac{C_1}{a}.$$
         We may assume $q\geq p>1.$ Set
        $$
        b:=\frac{1}{2}\left(\max\left\{0,\frac{n-2p}{p(n-2)}\right\}+\min\left\{\frac{n}{q(n-2)},1\right\}\right),
        $$
        then by assumption, we have
        $$
        0<\frac{(n-2)pb+2p-n}{p-1}\leq n,\ 0<n-(n-2)qb\leq n.
        $$
        Choose $a,q_0,r_0$ such that
        $$
        qa=p,\ q_0(1-a)=p,\ \frac{1}{q}+\frac{1}{q_0}+\frac{1}{r_0}=1,
        $$
        where if $p=q$, then $q_0=\infty$.
         Fix any $f\in C^0(\overline{M})$, then for any $ x\in M,$ we have 
        \begin{align*}
        & |B_M f(x)|^q \\
        =&\left[\int_{\Omega}|f(y)|^{a}|G(x,y)|^{b}|f(y)|^{1-a}| G(x,y)|^{1-b}dV(y)\right]^{q}  \\
        \leq &\int_{M}|f(y)|^{p}|G(x,y)|^{qb}dV(y)\cdot \left[\int_{M}|f(z)|^pdV(z)\right]^{\frac{q-p}{p}}  \\
        &\cdot \left[\int_{M}|G(x,w)|^{\frac{p(1-b)}{p-1}}dV(w)\right]^{\frac{(p-1)q}{p}}\\
        \leq &N\left(\frac{(n-2)pb+2p-n}{p-1},M\right)^{\frac{q(p-1)}{p}}\|f\|_{L^p(M)}^{q-p}\int_{M}|f(y)|^{p}| G(x,y)|^{qb}dV(y),
        \end{align*}
        so 
        \begin{align*}
        \|B_M f\|_{L^q(M)}^q&\leq N(qb(2-n)+n,M)N\left(\frac{(n-2)pb+2p-n}{p-1},M\right)^{\frac{q(p-1)}{p}} \\
        &\leq \frac{C_1}{qb(2-n)+n}\left(\frac{C_1(p-1)}{(n-2)pb+2p-n}\right)^{\frac{q(p-1)}{p}}\|f\|_{L^p(M)}^{q}.
        \end{align*}
     {\bf Case 2:} $n=2$.\\
     \indent If $q>1$, then $p>1$, we set $s:=\frac{p}{p-1}$. By Lemma \ref{lem:gradient 1}, there is a constant $C_2:=C_2(\overline M,g)>1$ such that 
         $$\sup_{x\in M}\int_{M}|G(x,y)|^{s}dV(y)\leq s^{s}C_2^{s+1}.$$
         By H\"older's Inequality, for any $x\in M$, and any $f\in C^0(\overline M)$, we have 
        $$|B_M f(x)|^q\leq  \left[\int_M|G(x,y)|^sdV(y)\right]^{\frac{q}{s}}\|f\|_{L^p(M)}^q\leq s^{q}C_2^{2q}\|f\|_{L^p(M)}^q.$$
       Thus, we know 
       $$\|B_M f\|_{L^q(M)}\leq \vol(M)^{\frac{1}{q}}sC_2^2\|f\|_{L^p(M)}.$$
     \indent If $q=1$, then we may assume $p=1$. By Lemma \ref{lem:gradient 1}, there is a constant $C_3:=C_3(\overline M,g)>1$ such that 
         $$\sup_{x\in M}\int_{M}|G(x,y)|dV(y)\leq C_3.$$
     By Fubini's Theorem, we have 
     $$\|B_M f\|_{L^1(M)}\leq C_3\|f\|_{L^1(M)}.$$
     \indent The proof is complete.
    \end{proof}
    For the proof of Theorem \ref{thm:laplace compact1},  we can use the same techniques as in the proof of Theorem \ref{thm:graident boundary} and \ref{thm:laplace boundary} since we have Lemma \ref{green function:compact} and \ref{lem:laplace 1}. We omit the details here.
    \section{The proofs of Theorem \ref{thm:green riesz}
    and \ref{thm:green riesz1}}\label{section:green riesz1}
    In this section, we give the proofs of Theorem \ref{thm:green riesz} and \ref{thm:green riesz1}. Many simple properties of subharmonic functions on Riemannian manifolds can be easily got from the article 
    \cite{Bon} via some minor modifications, so we don't repeat the arguments.
    \begin{thm}[= Theorem \ref{thm:green riesz}]
    Let $(M,g)$ be a Riemannian manifold without boundary (of dimension $\geq 2$), $\Omega\subset\subset M$ be an open subset
    with smooth boundary, and let $G$ be the Green function for the Laplace operator on $\overline\Omega$. Then for any function $f\in \operatorname{QSH}(M)$, we have
    $$
    f(x)=\int_{\Omega}G(x,y)\Delta f(y)+\int_{\partial \Omega}f(y)\frac{\partial G(x,y)}{\partial\nve_y}dS(y),\ \forall x\in\Omega.
    $$
    \end{thm}
    \begin{proof}
    Using Lemma \ref{Aubin:Theorem 4.8}, we may find a function $h\in C^\infty(\overline\Omega)$ which is strictly subharmonic 
    in an open neighborhood $U\subset\subset M$ of $\overline\Omega$, then for any $f\in \operatorname{QSH}(M)$, there exist $f_1\in C^2(U)$ and $f_2\in \sh(U)$ such that $f|_{U}=f_1+f_2$ (by a partition of unity). Clearly, $\Delta f=\Delta f_1+\Delta f_2,$ where $\Delta f_1$ is the signed measure determined by 
    $f_1$. Thus, by Lemma \ref{green function:noncompact}, we may assume $f\in \sh(M)$.\\
    \indent Choose an open neighborhood   $U\subset\subset M$ of $\overline\Omega$ with smooth boundary. By Lemma \ref{green function:noncompact}, for any fixed $x\in U$, we may choose a function $\Gamma(x,\cdot)$ such that 
    $$\Delta_y^{\distr}\Gamma(x,y)=\delta_x(y)\text{ in }U,$$
    and $\Gamma(x,y)=\Gamma(y,x)$ for all $x,y\in U,\ x\neq y$.
    By Remark \ref{remark:continuous},  we may find a function $\Phi\in C^0(\overline\Omega\times\overline \Omega)$ such that 
    for any fixed $x\in U$, we have 
     $$\left\{\begin{array}{ll}
     \Delta_y\Phi(x,y)=0 &\text{ in } \Omega,\\
     \Phi(x,y)=\Gamma(x,y) & \text{ on }\partial \Omega.
     \end{array}\right.$$
     By the uniqueness of the Green function, it is easy to see that $G(x,y)=\Gamma(x,y)-\Phi(x,y)$ for any $x,y\in\overline\Omega,\ x\neq y$.
     For any fixed $x\in \overline\Omega$, $\Phi(\cdot,x)=\Phi(x,\cdot)$ is a harmonic function on $\Omega$.
    It is obvious that 
    $$\Delta_x^{\distr}\left(f(x)-\int_{\Omega}\Gamma(x,y)\Delta f(y)\right)=0\text{ in }U.$$
    By Weyl's Lemma, there is a harmonic function $h$ on $U$ such that 
    $$f(x)=\int_{\Omega}\Gamma(x,y)\Delta f(y)+h(x)$$
    for almost every $x\in U$. However, for a subharmonic function, its value at every point is uniquely determined, so 
    we know 
    $$f(x)=\int_{\Omega}\Gamma(x,y)\Delta f(y)+h(x),\ \forall x\in U.$$
    Thus, we know 
    \begin{equation}\label{equ:qin}
    f(x)=\int_{\Omega}G(x,y)\Delta f(y)+l(x),\ \forall x\in \overline\Omega,
    \end{equation}
    where 
    $$l(x):=\int_{\Omega}\Phi(x,y)\Delta f(y)+h(x),\ \forall x\in \overline\Omega.$$
    Since $\Phi\in C^0(\overline\Omega\times\overline \Omega)$ and $\Phi(\cdot,y)$ is a harmonic function in $\Omega$
    for any $y\in\overline\Omega$, we know $l\in C^0(\overline\Omega)$ is harmonic in $\Omega$. Similar proofs with Green's identities and Lemma \ref{green function:noncompact} show that 
    $$l(x)=\int_{\partial \Omega}l(y)\frac{\partial G(x,y)}{\partial\nve_y}dS(y),\ \forall x\in\Omega.$$
    Combining Equality (\ref{equ:qin}) and note that $G|_{\Omega\times\partial\Omega}=0$, we know
    $$l(x)=\int_{\partial \Omega}f(y)\frac{\partial G(x,y)}{\partial\nve_y}dS(y),\ \forall x\in\Omega,$$
    which completes the proof.
    \end{proof}
    Using the main idea of the proof of Theorem \ref{thm:green riesz}, we now give the proof of Theorem \ref{thm:green riesz1}. 
    \begin{thm}[= Theorem \ref{thm:green riesz1}]
    Let $(M,g)$ be a compact Riemannian manifold without boundary, and let $G$ be the Green function of $M$. Then for any function $f\in \operatorname{QSH}(M)$, we have
    \begin{equation*}\label{equ:green-riesz1}
    f(x)=f_{\ave}+\int_{M}G(x,y)\Delta f(y),\ \forall x\in M.
    \end{equation*}
    \end{thm}
    \begin{proof}
    It is clear that 
    $$\Delta_x^{\distr}\left(f(x)-\int_{\Omega}G(x,y)\Delta f(y)\right)=0\text{ in }M.$$
    By Weyl's Lemma, there is a harmonic function $h$ on $M$ such that 
    $$f(x)=\int_{\Omega}G(x,y)\Delta f(y)+h(x)$$
    for almost every $x\in M$, then 
    $$f(x)=\int_{M}G(x,y)\Delta f(y)+h(x),\ \forall x\in M.$$
    Using Lemma \ref{green function:compact} and note that $\int_{M}G(x,y)dV(y)=0$ for all $x\in M$, we then know 
    $$h(x)=h_{\ave}=f_{\ave},\ \forall x\in M.$$
    \end{proof}
    
    \section{The proof of Theorem \ref{thm:poincare average1}}\label{section:L1}
     In this section, we give the proof of Theorem \ref{thm:poincare average1}. Firstly, the following weak $L^1$-estimate of the Hardy-Littlewood maximal function is well known, see \cite[Theorem 3.2.7]{Hao} for example.
     \begin{lem}\label{lem:maximal function}
     For any $f\in L^1(\mr^n)$, and any $t>0$, we have 
     $$\left|\{x\in\mr^n|\ (Mf)(x)>t\}\right|\leq \frac{3^n}{t}\|f\|_{L^1(\mr^n)},$$
      where 
     $$(Mf)(x):=\sup_{r>0}\frac{1}{|B(x,r)|}\int_{B(x,r)}f(y)dV(y).$$
    \end{lem}
    \indent Please see {\cite[Theorem 1.9.1]{Henkin}} or {\cite[Theorem 3.1]{Ada}} for the classical Bochner-Martinealla formula.
    Let $\lip_c(\mc^n)$ denote Lipschitz functions on $\mc^n$ which have compact supports,
    then we know the Bochner-Martinella formula still holds for any $f\in \lip_c(\mc^n)$ by approximation
    (note that $|\bar\partial f|\in L^\infty(\mc^n)$ by \cite[Theorem 4, Section 5.7]{Evans}),
    i.e., we have
    \begin{lem}\label{lem:Bochner-Martinella}
    For any $f\in \lip_c(\mc^n)$ and any $z\in\mc^n$ such that $f(z)\neq 0$, we have 
    $$
   f(z)=-\frac{(n-1)!}{(2\pi i)^n}\int_{\mc^n}\bar{\partial}f(\xi)\wedge \omega(\xi,z),
    $$
    where
    $$
    \omega(\xi,z):=\frac{1}{|\xi-z|^{2n}}\sum_{j=1}^n(-1)^{j+1}(\bar{\xi}_j-\bar{z}_j)
    \mathop{\wedge}_{k\neq j}d\bar{\xi}_k\wedge \mathop{\wedge}_{l=1}^nd\xi_l.
    $$
    \end{lem}
    \indent Now we give a key lemma which is important for the proof of Theorem \ref{thm:poincare average}, which we think has some independent interests.
    \begin{lem}
    There is a constant $\delta:=\delta(n)>0$ such that 
    $$\delta\|f\|_{L^{\frac{2n}{2n-1}}(\mc^n)}\leq \|\bar\partial f\|_{L^1(\mc^n)},\ \forall f\in C_c^1(\mc^n).$$
    \end{lem}
    \begin{proof}
    The proof is inspired by the proof of \cite[Theorem 3.6]{Kinnunen}.
         For any $j\in\mathbb{Z}$, set 
     $$A_j:=\{z\in\mc^n|\ 2^j<|f(z)|\leq 2^{j+1}\},$$
     and define 
     $$f_j\colon \mc^n\rw \mc,\ z\rwo 
     \left\{
     \begin{array}{ll}
     0 & \text{ if }|f(z)|\leq 2^{j-1},\\
     2^{1-j}f(z)-1 & \text{ if }2^{j-1}<|f(z)|<2^j,\\
     1 & \text{ if }|f(z)|>2^j.
     \end{array}\right.$$
     For any $j\in\mathbb{Z}$, we have $\bar\partial f_j=2^{1-j}\bar\partial f$ for almost every $z$ such that $2^{j-1}<|f(z)|<2^j$,
     and $\bar\partial f_j=0$ for almost any other $z$, then we know $f_j\in \lip_c(\mc^n)$.
     By Lemma \ref{lem:Bochner-Martinella} and note that $\left|\mathop{\wedge}_{k=1}^nd\bar{\xi}_k\wedge \mathop{\wedge}_{l=1}^nd\xi_l\right|= 2^ndV(\xi),$
     for any $z\in\supp(f)$, we have 
     $$
     |f_j(z)|\leq \frac{n!}{\pi^n}\int_{\mc^n}|\bar\partial f_j(\xi)|\cdot |\xi-z|^{1-2n}dV(\xi).$$
     We split this integral into two parts as $\int_{\mc^n}=\int_{B(z,r)}+\int_{\mc^n\setminus B(z,r)}$
     for some $r$ to be determined later. For the first part, we have 
    \begin{align*}
    &\quad \int_{B(z,r)}|\bar\partial f_j(\xi)|\cdot|\xi-z|^{1-2n}dV(\xi)\\
    &=\sum_{k=0}^\infty\int_{B(z,2^{-k}r)\setminus B(z,2^{-k-1}r)}|\bar\partial f_j(\xi)|\cdot|\xi-z|^{1-2n}dV(\xi)\\
    &\leq \sum_{k=0}^\infty\int_{B(z,2^{-k}r)\setminus B(z,2^{-k-1}r)}|\bar\partial f_j(\xi)|\cdot(2^{-k-1}r)^{1-2n}dV(\xi)\\
    &\leq \sum_{k=0}^\infty\int_{B(z,2^{-k}r)}|\bar\partial f_j|\cdot(2^{-k-1}r)^{1-2n}dV\\
    &\leq \sum_{k=0}^\infty |B(z,2^{-k}r)|\cdot(2^{-k-1}r)^{1-2n} M(|\bar\partial f_j|)(z)\\
    &=4^n|B(0,1)|r M(|\bar\partial f_j|)(z).
    \end{align*} 
    For the second part, we have 
    $$
    \int_{\mc^n\setminus B(z,r)}|\bar\partial f_j(\xi)|\cdot|\xi-z|^{1-2n}d\lambda(\xi)\leq \|\bar\partial f_j\|_{L^1(\mc^n)} r^{1-2n}.
    $$
    Choose $r$ such that 
    $$4^n|B(0,1)|r M(|\bar\partial f_j|)(z)=\|\bar\partial f_j\|_{L^1(\mc^n)}r^{1-2n},$$
    then we have 
    $$|f_j(z)|\leq \frac{2n!}{\pi^n}\left(4^n|B(0,1)| M(|\bar\partial f_j|)(z)\right)^{\frac{2n-1}{2n}}\|\bar\partial f_j\|_{L^1(\mc^n)}^{\frac{1}{2n}},
    \forall z\in \supp(f).$$
    By Lemma \ref{lem:maximal function}, we have 
     \begin{align*}
      |A_j|&\leq |\{z\in \mc^n|\ |f(z)|>2^{j}\}|=|\{z\in\mc^n|\ f_j(z)=1\}|\\
      &\leq \left|\left\{z\in\mc^n\left|\  \frac{2n!}{\pi^n}\right.\left(4^n|B(0,1)| M(|\bar\partial f_j|)(z)\right)^{\frac{2n-1}{2n}}\|\bar\partial f_j\|_{L^1(\mc^n)}^{\frac{1}{2n}}
      \geq 1\right\}\right|\\
      &\leq \left(\frac{2n!}{\pi^n}\right)^{\frac{2n}{2n-1}}36^n|B(0,1)|\|\bar\partial f_j\|_{L^1(\mc^n)}^{\frac{2n}{2n-1}}\\
      &\leq \left(\frac{2n!}{\pi^n}\right)^{\frac{2n}{2n-1}}36^n|B(0,1)|\left(\int_{2^{j-1}<|f(z)|<2^j}2^{1-j}|\bar\partial f|d\lambda\right)^{\frac{2n}{2n-1}}\\
      &=\left(\frac{2n!}{\pi^n}\right)^{\frac{2n}{2n-1}}36^n|B(0,1)|2^{\frac{2n(1-j)}{2n-1}}\left(\int_{A_{j-1}}|\bar\partial f|d\lambda\right)^{\frac{2n}{2n-1}}.
     \end{align*}
     Then we have 
     \begin{align*}
     \int_{\mc^n}|f|^{\frac{2n}{2n-1}}d\lambda&=\sum_{j\in\mathbb{Z}}\int_{A_j}|f|^{\frac{2n}{2n-1}}d\lambda
     \leq \sum_{j\in\mathbb{Z}}2^{\frac{2n(j+1)}{2n-1}}|A_j|\\
    &\leq \left(\frac{2n!}{\pi^n}\right)^{\frac{2n}{2n-1}}36^n|B(0,1)|\sum_{j\in\mathbb{Z}}2^{\frac{4n}{2n-1}}\left(\int_{A_{j-1}}|\bar\partial f|d\lambda\right)^{\frac{2n}{2n-1}}\\
    &=\left(\frac{8n!}{\pi^n}\right)^{\frac{2n}{2n-1}}36^n|B(0,1)|\left(\int_{\mc^n}|\bar\partial f|d\lambda\right)^{\frac{2n}{2n-1}},
     \end{align*}
     and the proof is complete.
    \end{proof}
    Similar to the proof of \cite[Theorem 2.6]{Hebey} (use partition of unity), we have 
    \begin{cor}\label{cor:average1}
    Let $(M,h)$ be a compact Hermitian manifold of complex dimension $n$, then there is a constant $\delta:=\delta(M,h)>0$ such that 
    $$\delta\|f\|_{L^{\frac{2n}{2n-1}}(M)}\leq \|\bar\partial f\|_{L^1(M)}+\|f\|_{L^1(M)},\ \forall f\in C^1(M).$$
    \end{cor}
    \begin{lem}\label{lem:average2}
    Let $(M,h)$ be a compact balanced manifold, then there is a constant $\delta:=\delta(M,h)>0$ such that 
    $$\delta\|f-f_{\ave}\|_{L^1(M)}\leq \|\bar\partial f\|_{L^1(M)},\ \forall f\in C^1(M).$$
    \end{lem}
    \begin{proof}
    We may assume $f\in C^2(M)$, and let $G$ be a Green function as in Lemma \ref{green function:compact} on the underlying 
    Riemannian manifold $(M,g)$. By Lemma \ref{lem:laplace 1}, there is a constant $C:=C(M,h)>0$ such that 
    $$\sup_{z\in M}\int_M|\bar\partial_wG(z,w)|dV(w)\leq C.$$
    By Lemma \ref{green function:compact}, using an integration by parts, we have 
    $$f(z)-f_{\ave}=-2\int_{M}\langle\bar\partial_w G(z,w),\bar\partial f(w)\rangle dV(w),\ \forall z\in M.$$
    By Fubini's Theorem, we get 
    $$\|f-f_{\ave}\|_{L^1(M)}\leq 2C\|\bar\partial f\|_{L^1(M)}.$$
    \end{proof}
    \indent Now we can prove Theorem \ref{thm:poincare average1}.
     \begin{thm}[= Theorem \ref{thm:poincare average1}]
    Let $(M,h)$ be a compact balanced manifold of complex dimension $n$, and let 
    $$1\leq q\leq \frac{2n}{2n-1}.$$
     Then there is a constant $\delta:=\delta(M,h,q)>0$ such that
    $$
    \delta\|f-f_{\operatorname{ave}}\|_{L^q(M)}\leq \|\bar\partial f\|_{L^1(M)},\ \forall f\in C^1(M).$$
    \end{thm}
    \begin{proof}
     We may assume $q=\frac{2n}{2n-1}$, and fix $f\in C^1(M)$.
     By Corollary \ref{cor:average1} and Lemma \ref{lem:average2}, there is a constant $\delta:=\delta(M,h)>0$ such that 
    $$\delta\|f-f_{\ave}\|_{L^{\frac{2n}{2n-1}}(M)}\leq \|\bar\partial f\|_{L^1(M)}+\|f-f_{\ave}\|_{L^1(M)},$$
    and 
    $$\delta\|f-f_{\ave}\|_{L^1(M)}\leq \|\bar\partial f\|_{L^1(M)}.$$
    Thus we have 
    $$\frac{\delta^2}{\delta+1}\|f-f_{\ave}\|_{L^{\frac{2n}{2n-1}}(M)}\leq \|\bar\partial f\|_{L^1(M)}.$$
    \end{proof}
    \section{The proof of Theorem \ref{thm:improved L2 estimate}}
    Assumptions and notations as in Theorem \ref{thm:improved L2 estimate}. Let $\dom(\bar\partial^E)$ and $\dom(\bar\partial^{E,*})$ denote the domain 
    of $\bar\partial^E$ and $\bar\partial^{E,*}$, respectively.  \\
    \indent By a partition of unity, the following lemma can be easily derived from \cite[Proposition 2.1.1]{Hor65}
    \begin{lem}\label{approximation}
     $C_{(0,1)}^1(\overline{\Omega},E)\cap \dom(\bar\partial^{E,*})$ is dense in $\dom(\bar\partial^E)\cap \dom(\bar\partial^{E,*})$ under the graph norm
     given by 
     $$\alpha\mapsto \|\alpha\|+\|\bar\partial^E\alpha\|+\|\bar\partial^{E,*}\alpha\|.$$
    \end{lem}
    The following global Bochner-Hormander-Kohn-Morrey formula is well known, see \cite[Theorem 1.4.21]{Ma} for example.
    \begin{lem}\label{lem:bochner}
    Let $(M,\omega)$ be a K\"ahler manifold, $\Omega\subset M$ be a relatively compact open subset with a $C^2$-boundary defining function $\rho$, and let $(E,h)$ be a Hermitian holomorphic vector bundle over $M$. Then for any $\alpha\in C_{(n,1)}^1(\overline \Omega,E)\cap \dom(\bar\partial^{E,*})$, we have 
    $$\|\bar\partial^E \alpha\|^2+\|\bar\partial^{E,*}\alpha\|^2=\|\partial^{E,*} \alpha\|^2+(i\Theta^{(E,h)}\wedge \Lambda \alpha,\alpha)+\int_{\partial M}L_\rho(\alpha,\alpha)dS,$$
    where $L_\rho$ is the Levi form of $\rho$.
    \end{lem}
     Now we prove Theorem \ref{thm:improved L2 estimate}.
     \begin{thm}[= Theorem \ref{thm:improved L2 estimate}]
     Let $(M,\omega)$ be a K\"ahler manifold,  $\Omega\subset\subset M$ be an open subset with a smooth strictly plurisubharmonic boundary defining function $\rho$, and let $(E,h)$ be a Hermitian holomorphic vector bundle over $M$ such that $A_{E}:=i\Theta^{(E,h)}\wedge \Lambda\geq 0$ in bidegree $(n,1)$. Then there is a constant $\delta:=\delta(M,\omega,\Omega,\rho,E,h)>0$
     such that for any nonzero $\bar\partial$-closed $f\in L^2_{(n,1)}(\Omega,E),$ satisfying 
     $$M_f:=\int_{\Omega}\langle A_E^{-1}f,f\rangle dV<\infty,$$
     there is $u\in L_{(n,0)}^2(\Omega,E)$ such that  $\bar\partial u=f$ and 
     $$\int_{\Omega}|u|^2dV\leq \frac{\|f\|}{\sqrt{\|f\|^2+\delta M_f}}\int_{\Omega}\langle A_E^{-1}f,f\rangle dV.$$
    \end{thm}
     \begin{proof}
    Set $M_f:=\int_{\Omega}\langle A_E^{-1}f,f\rangle dV$. Similar to the proof of \cite[Theorem 4.5, Chapter VIII]{Dem}, by the Hahn-Banach extension theorem and Riesz representation theorem, it suffices to prove there is a $\delta:=\delta(M,\omega,\Omega,\rho,E,h)>0$
    such that for any $(n,1)$-form $\alpha\in \dom(\bar\partial^{E})\cap\dom(\bar\partial^{E,*})$, we have
    $$
    |(\alpha,f)|^2+\delta\frac{|(\alpha,f)|^2}{\|f\|_2^2} M_f\leq \left(\|\bar\partial^E\alpha\|^2+\|\bar\partial^{E,*}\alpha\|^2\right)M_f.
    $$
    By Lemma \ref{approximation}, we may assume $\alpha\in C^1_{(0,1)}(\overline{\Omega})$. By Theorem \ref{thm: dbar boundary} and since $\Omega$ is strictly pseudoconvex,
    using a partition of unity, there is a constant $\delta:=\delta(M,\omega,\Omega,\rho,E,h)>0$  such that
        $$
        \delta\|\alpha\|^2\leq \|\partial^{E,*}\alpha\|^2+\int_{\partial \Omega}L_\rho(\alpha,\alpha)dS.
        $$
     Using Cauchy-Schwarz Inequality and Lemma \ref{lem:bochner}, we get
    \begin{align*}
    &\quad |(\alpha,f)|^2+\delta\frac{|(\alpha,f)|^2}{\|f\|_2^2} M_f\\
    &\leq (A_E\alpha,\alpha)\cdot M_f+\delta\|\alpha\|^2\cdot M_f\\
    &\leq \left[(A_E\alpha,\alpha)+\|\partial^{E,*}\alpha\|^2+\int_{\partial M}L_\rho(\alpha,\alpha)dS\right]\cdot M_f\\
    &\leq \left(\|\bar\partial^E\alpha\|^2+\|\bar\partial^{E,*}\alpha\|^2\right)M_f.
    \end{align*}
    \end{proof}
    \bibliographystyle{amsplain}
        
\end{document}